\documentclass[11pt,twoside,reqno]{amsart}
\usepackage{amsmath, amsfonts, amsthm, amssymb,amscd, graphicx, amscd}
\usepackage{float,epsf}
\usepackage[english]{babel}
\usepackage{enumerate}
\usepackage{tikz}
\usepackage{hyperref}
 \usepackage{float} 
\usepackage{srcltx}
\usepackage{subfigure}
\usepackage{graphics}
\usepackage{epsfig}

\usepackage{geometry}
\usepackage{verbatim}
\usepackage{mathrsfs}

\newtheorem{theorem}{Theorem}[section]
\newtheorem{lemma}[theorem]{Lemma}

\newtheorem{proposition}[theorem]{Proposition}
\newtheorem{corollary}[theorem]{Corollary}
\newtheorem{remark}[theorem]{Remark}

\newtheorem{conjecture}{Conjecture}

\numberwithin{equation}{section}
\numberwithin{figure}{section}

\def\intave#1{\int_{#1}\hbox{\llap{$\raise2.3pt\hbox{\vrule
height.9pt width7pt}\phantom{\scriptstyle{#1}}\mkern-2mu$}}}

\title{On Laplacian eigenvalue equation with constant Neumann boundary data}
\usepackage[misc]{ifsym}

\author{Yong Huang}
\address{School of Mathematics, Hunan University, Changsha, Hunan, China.}
\email{huangyong@hnu.edu.cn}

\author{Qinfeng Li
}
\address{School of Mathematics, Hunan University, Changsha, Hunan, China.}
\email{liqinfeng@hnu.edu.cn}

\author{Qiuqi Li }
\address{School of Mathematics, Hunan University, Changsha, Hunan, China.}
\email{qli28@hnu.edu.cn}

\author{Ruofei Yao}
\address{School of Mathematics, South China University of Technology, Guangzhou, Guangdong, China.}
\email{yaoruofei@scut.edu.cn}

\thanks{Research of Yong Huang was supported by the National Science Fund (12171144, 12231006). Research of Qinfeng Li was supported by National Key R\&D Program of China (2022YFA1006900) and the National Natural Science Fund for Youth Scholars (No. 12101215). Research of Qiuqi Li was supported by the National Natural Science Fund for Youth Scholars (No. 12101216). Research of Ruofei yao was supported by the National Natural Science Fund for Youth Scholars (No. 12001543).}

\begin{document}

\begin{abstract}
Let $\Omega$ be a bounded Lipschitz domain in $\mathbb{R}^n$ and we study boundary behaviors of solutions to the following Laplacian eigenvalue equation with constant Neumann data:
 \begin{align}
    \label{cequation0}
\begin{cases}
-\Delta u=cu\quad &\mbox{in $\Omega$}\\
\frac{\partial u}{\partial \nu}=-1\quad &\mbox{on $\partial \Omega$}.
\end{cases}    
\end{align}First, by using properties of Bessel functions and proving new inequalities on elementary symmetric polynomials, we obtain the following inequality for rectangular boxes, balls and equilateral triangles:
\begin{align}
\label{bbb}
    \lim_{c\rightarrow \mu_2^-}c\int_{\partial \Omega}u_c\, \rm{d} \sigma\ge \frac{n-1}{n}\frac{P^2(\Omega)}{|\Omega|},
\end{align}with equality achieved only at cubes and balls. In the above, $u_c$ is the solution to \eqref{cequation0} and $\mu_2$ is the second Neumann Laplacian eigenvalue. Second, let $\kappa_1$ be the best constant for the Poincar\'e inequality with the vanishing mean condition over $\partial \Omega$, and we prove that $\kappa_1\le \mu_2$ and that the equality holds if and only if $\int_{\partial \Omega}u_c\, \rm{d} \sigma>0$ for any $c\in (0,\mu_2)$. As a consequence, $\kappa_1=\mu_2$ on balls, rectangular boxes and equilateral triangles, and balls maximize $\kappa_1$ over all Lipschitz domains with fixed volume. As an application, we extend the symmetry breaking results from ball domains obtained in  Bucur-Buttazzo-Nitsch[J. Math. Pures Appl. (9) 107 (2017), no. 4, 451–463.], to wider class of domains, and give quantitative estimates for the precise breaking threshold at balls and rectangular boxes. It is a direct consequence that for domains with $\kappa_1<\mu_2$, \eqref{bbb} is never true, while whether it is valid for domains on which $\kappa_1=\mu_2$ remains open.
\end{abstract}

\maketitle

\vskip0.2in

{\bf Key words}: Spectral Theory, Bessel Fucntions, Isoperimetric Inequality, Neumann Laplacian

{\bf 2020 AMS subject classification}: 49K20, 49K40, 49R05, 53E99, 34B30

\section{Introduction} 
In this paper, our focus is initially driven by an examination of the following shape functional on a bounded Lipschitz domain $\Omega$ in $\mathbb{R}^n$:
\begin{align}
\label{defofkappa}
    \kappa_1(\Omega)=\inf\Big\{\frac{\int_{\Omega}|\nabla u|^2\, dx}{\int_{\Omega}u^2\, dx}: u \in H^1(\Omega)\setminus\{0\},\, \int_{\partial \Omega}u\,\rm{d} \sigma=0\Big\}.
\end{align} The quantity $\kappa_1=\kappa_1(\Omega)$ is the best constant for the Poincar\'e type inequality with a vanishing mean over $\partial \Omega$.  It is noteworthy that minimizing the Rayleigh quotient in \eqref{defofkappa} under the zero boundary condition corresponds to determining the first eigenvalue of the Dirichlet Laplacian, denoted by $ \lambda_1$. Similarly, under the vanishing mean condition over $\Omega$, it corresponds to finding the second eigenvalue of the Neumann Laplacian, denoted by $\mu_2$, and similar problems are considered in \cite{BP00}, \cite{FH04}, \cite{GL08}, etc. For comprehensive coverage of Laplacian eigenvalues with various boundary conditions, we refer to the classical monograph \cite{PS} and two recent exhaustive monographs \cite{Henrot2} and \cite{nshap} and references therein. Regarding Poincar\'e type inequalities, we refer to the classical books \cite{AH} and \cite{M}, and for applications in shape optimization we refer to the monographs \cite{BB} and \cite{Henrot}.

It seems to us that very few literature addresses \eqref{defofkappa}, and so far the exact constant $\kappa_1$ is only obtained for rectangular boxes, see \cite{NP15}. Determining the value of $\kappa_1$ and comparing it with $\mu_2$ on generic domains is important in its own right, as it relates to the best constant in a Poincar\'e type inequality. Furthermore, the concept of $\kappa_1$ plays a crucial role in the perturbation argument for obtaining the symmetry breaking result proved in \cite{BBN} on a thin thermal insulation problem, see also the related works in \cite{BBN1}, \cite{Buttazzo}, \cite{PNST}, \cite{HLL}, \cite{HLL2} and \cite{PNT}. From these motivations, we began our investigation of the shape functional defined in \eqref{defofkappa}.

It turns out that it is more essential to study the following Laplacian eigenvalue equation
\begin{align}
    \label{cequation}
\begin{cases}
-\Delta u=cu\quad &\mbox{in $\Omega$}\\
\frac{\partial u}{\partial \nu}=-1\quad &\mbox{on $\partial \Omega$},
\end{cases}    
\end{align}where $c$ is a positive constant. On one hand, \eqref{cequation} is closely associated with $\kappa_1$, and by delving into the study of \eqref{cequation}, we can characterize exact values of $\kappa_1$ and compare it with $\mu_2$. On the other hand, there may be a novel geometric inequality underlying \eqref{cequation} that links $\kappa_1$, $\mu_2$ and the isoperimetric ratio for specific domain categories, and this provides a new perspective in understanding the symmetry breaking phenomenon over balls proved in \cite{BBN}. It can also be applied to prove similar phenomenon over more general domains, obtain quantitative results and extend previous researches in \cite{BBN}, \cite{HLL} and \cite{HLL2}. These will be further explained in the later part of the paper.

Hence, from now on, the present paper will focus on \eqref{cequation}, and properties of $\kappa_1$ will be derived as consequences. As a first step, we should state the existence result of solutions to \eqref{cequation} for clarity. Let $0=\mu_1<\mu_2\le \cdots \le \mu_k\rightarrow \infty$ be the Neumann Laplacian eigenvalues on $\Omega$ enumerated in an increasing order and repeated according to their multiplicity. It is well known that if $c$ does not belong to any of the eigenvalues $\mu_k$, then \eqref{cequation} admits a unique weak solution in $H^1(\Omega)$. Concerning singular cases, we have the following theorem.
\begin{theorem}
\label{existence}
Let $\mu_k$ be the $k$-th eigenvalue of Neumann Laplacian. Then when $c=\mu_k$, $k \ge 2$, \eqref{cequation} admits a weak solution in $H^1(\Omega)$ if and only if $\Omega$ satisfies the following property:
\begin{align}
    \label{omegaproperty}
\int_{\partial \Omega} v_k \, \rm{d} \sigma=0,\quad \mbox{for every $v_k \in E_k(\Omega)$,}    
\end{align}
where $E_k(\Omega)$ stands for the Neumann Laplacian eigenspace of $\mu_k$ on $\Omega$.
\end{theorem}

The proof of Theorem \ref{existence} is quite elementary, but what makes the result interesting is that it tells us that the existence of solutions on symmetric domains is partially related to anti-symmetry of Neumann Laplacian modes. For example, when $c=\mu_2$, \eqref{cequation} is solvable when $\Omega$ is a convex domain with two axes of symmetry, or $\Omega$ is an isosceles triangle with aperture greater than $\pi/3$, since it is known due to \cite{Friedlander} and \cite{LS10} that in these cases any Neumann eigenfunctions of $\mu_2$ must be antisymmetric with respect to an axis. Note that for equilateral triangles, even if Neumann Laplacian eigenfunctions of $\mu_2$ contain symmetric modes, the symmetric modes also satisfy \eqref{omegaproperty}, and hence \eqref{cequation} also admits a solution when $c=\mu_2$. 

Given $k\ge 2$, one may ask what the common features of domains are such that \eqref{omegaproperty} is satisfied. This is a very difficult question to answer even for special domains and $k=2$. For example, it is not an easy task to show that any isosceles triangle with aperture strictly less than $\pi/3$ cannot satisfy \eqref{omegaproperty} when $k=2$, even though it is known that modes of $\mu_2$ on these domains cannot be anti-symmetric, due to \cite{LS10}. We leave the question for future work. 

Now we describe our main results, which consist of three parts. The first part is on isoperimetric type inequalities related to \eqref{cequation} on special domains. The second part is on the relation between \eqref{cequation} and the comparison of $\kappa_1$ and $\mu_2$. The third part discusses applications of these results to a thermal insulation problem.

\subsection{Isoperimetric type inequalities on special domains}

We start with analyzing \eqref{cequation} on special domains. Throughout the paper, we use $P(\cdot)$ to denote the perimeter of a set, and $|\cdot|$ to denote the volume of a set. The first result we have is the following.
\begin{theorem}
\label{zheng1}
Let $\Omega$ be a bounded Lipschitz domain in $\mathbb{R}^n$, $c \in (0,\mu_2)$ and $u_c$ be the solution to \eqref{cequation}. Then 
\begin{align}
\label{quyu0}
    \lim_{c\rightarrow 0}c\int_{\partial \Omega}u_c\, \rm{d} \sigma =\frac{P^2(\Omega)}{|\Omega|}.
\end{align}
Moreover, if $\Omega$ is a ball of radius $R$, then
\begin{align}
\label{ballcase}
    \lim_{c\rightarrow \mu_2^-}c\int_{\partial \Omega}u_c\, \rm{d} \sigma =\frac{n-1}{n}\frac{P^2(\Omega)}{|\Omega|}(=2\pi, \, \mbox{when $n=2$}).
\end{align}
\end{theorem}
Both \eqref{quyu0} and \eqref{ballcase} are rather intriguing to us because they link \eqref{cequation} to the isoperimetric ratio $P^2(\Omega)/|\Omega|$. The formula \eqref{quyu0} is obtained as a byproduct in studying asymptotic behavior of solutions to a variational problem introduced in later sections, and \eqref{ballcase} is proved by making use of recurrence formulas of Bessel functions. A different proof of \eqref{ballcase} was actually implicitly given in \cite{HLL}, via more geometric yet very complicated argument involving calculating second derivatives of shape functionals.

Motivated by \eqref{ballcase} in the case of ball domains, we then estimate the quantity $$\lim_{c \rightarrow \mu_2^-}c\int_{\partial \Omega}u_c \, \rm{d} \sigma$$ among rectangular boxes in $\mathbb{R}^n$, which we mean by Cartesian products of bounded open intervals. We are particularly focused on the comparison of the above limit  with the quantity $\tfrac{n-1}{n}\tfrac{P^2(\Omega)}{|\Omega|}$. By playing several inequalities on elementary symmetric polynomials, we derive the following theorem. 
\begin{theorem}
\label{zheng2}
Let $c \in (0,\mu_2)$ and $u_c$ be the solution to \eqref{cequation}. Then among the class of rectangular boxes $\Omega \subset \mathbb{R}^n$, we have
    \begin{align}
    \label{jianrenzhichang}
        \lim_{c\rightarrow \mu_2}c\int_{\partial \Omega}u_c\, \rm{d} \sigma \ge \frac{n-1}{n}\frac{P^2(\Omega)}{|\Omega|}.
    \end{align}
The equality in \eqref{jianrenzhichang} is achieved only when $\Omega$ is a cube in $\mathbb{R}^n$.
\end{theorem}
The proof of \eqref{jianrenzhichang} is equivalent to proving the following elementary inequality involving both trigonometric functions and $\sigma_k$ operators:
\begin{align}
\label{nandian}
    \sum_{i=2}^n2^{n-1}\pi \frac{\sigma_n}{a_1a_i}\frac{1}{\tan (\frac{\pi}{2}\frac{a_i}{a_1})}+2^{n+1}\sigma_{n-2}\ge \frac{n-1}{n} 2^n \frac{\sigma_{n-1}^2}{\sigma_n},
\end{align}
where $a_1\ge \cdots \ge a_n>0$, $\sigma_k$ is denoted as the $k$-th elementary symmetric polynomials in variables $a_1, \cdots, a_n$, and the equality is achieved only when $a_1=\cdots=a_n$. Note that in general, the second term in the left hand side of \eqref{nandian} is less than or equal to the right hand side due to Newton's inequalities, and hence the nontrivial part in the proof lies in dealing with the first term.

As a consequence of Theorem \ref{zheng2} and Maclaurin's inequality, one can immediately see that among rectangular domains with prescribed perimeter, cubes are the unique minimizers to the left hand side of \eqref{jianrenzhichang}.


After establishing Theorem \ref{zheng1} and Theorem \ref{zheng2}, one may conjecture that \eqref{jianrenzhichang} could be true for general Lipschitz domains. Unfortunately, results stated in the next subsection actually disprove the conjecture, by demonstrating that \eqref{jianrenzhichang} is always false for domains satisfying $\kappa_1<\mu_2$, where $\kappa_1$ is as in \eqref{defofkappa}. It remains open whether \eqref{jianrenzhichang} is valid on domains with $\kappa_1=\mu_2$.

Nevertheless, numerical results suggest that \eqref{jianrenzhichang} might be true for convex domains satisfying $\kappa_1=\mu_2$, and all of such domains so far we have found have some symmetric properties. For example, in section 8 we will provide abundant numerical results on regular polygons, super-equilateral triangles, ellipses, rhombuses and so on, all of which satisfy \eqref{jianrenzhichang}. Our numerical results suggest several interesting phenomena which have not been found in literature, and can provide valuable inspirations for future researches. Notably, some of the numerical results highly support the following conjecture:
\begin{conjecture}
\label{conj2}
Let $\Omega$ be a convex domain in $\mathbb{R}^2$ such that $\kappa_1=\mu_2$. Then 
\begin{align}
\label{planar}
    \lim_{c\rightarrow \mu_2} c\int_{\partial \Omega}u_c\, \rm{d} \sigma \ge \frac{1}{2}\frac{P^2(\Omega)}{|\Omega|},
\end{align}with equality holding if and only if $\Omega$ is a disk or a regular polygon in $\mathbb{R}^2$ with $k$ sides, $k \ge 4$.
\end{conjecture}

Surprisingly, the equality above cannot be achieved at equilateral triangles, due to explicit formulas of solutions on equilateral triangles, see section 6. Conjecture \ref{conj2} is open to us, and particularly it deserves further consideration regarding the equality principle of \eqref{planar}.


\subsection{Eigenvalue comparison results}

The results described in this subsection turn to the comparison of $\kappa_1$ and $\mu_2$, which can be derived from analysis of \eqref{cequation}, and is also partially motivated from Theorem \ref{zheng1} and Theorem \ref{zheng2} above.

First, it is easy to prove that $\kappa_1 \le \mu_2$ by a suitable choice of trial functions, but it is a challenging task to geometrically classify domains such that $\kappa_1=\mu_2$.  In general, proving whether or not $\kappa_1=\mu_2$ on special domains is a nontrivial problem, and even for rectangular boxes, it takes very sophisticated argument to conclude this fact, see \cite{NP15}. Nevertheless, we obtain in the following theorem a necessary and sufficient condition such that $\kappa_1=\mu_2$, through the analysis of the equation \eqref{cequation}. As a byproduct, it disproves  \eqref{jianrenzhichang} for domains with $\kappa_1<\mu_2$.
\begin{theorem}
\label{maintheorem}
Let $\Omega$ be a bounded Lipschitz domain in $\mathbb{R}^n$, $\kappa_1$ be defined as in \eqref{defofkappa} and $\mu_2$ be the second Neumann Laplacian eigenvalue on $\Omega$. Then $\kappa_1\le \mu_2$, and the equality holds if and only if 
\begin{align}
\label{iff}
    c\int_{\partial \Omega}u_c\, \rm{d} \sigma >0, \quad \mbox{\rm{for any} $c \in (0,\mu_2)$,}
\end{align}where $u_c$ is the solution to \eqref{cequation}.
Moreover, if $\kappa_1<\mu_2$, then
\begin{align}
\label{zonghe}
    \begin{cases}
    c\int_{\partial \Omega}u_c\,\rm{d} \sigma>0,\quad &\mbox{\rm{for any} $c\in (0,\kappa_1)$}\\
    c\int_{\partial \Omega}u_c\, \rm{d} \sigma<0, \quad &\mbox{\rm{for any} $c\in (\kappa_1,\mu_2)$}\\
    \lim_{c \rightarrow \kappa_1} c\int_{\partial \Omega}u_c\, \rm{d} \sigma=\kappa_1\int_{\partial \Omega}u_{\kappa_1}\, \rm{d} \sigma=0.
    \end{cases}
\end{align}
\end{theorem}


For the critical case $c=\mu_2$, we also have the following result.
\begin{proposition}
\label{byproduct1}
Let $\Omega$ be a bounded Lipschitz domain. If $\kappa_1=\mu_2$, then \eqref{cequation} admits a solution $u_{\mu_2}$ when $c=\mu_2$. Conversely, if $\eqref{cequation}$ admits a solution $u_{\mu_2}$ when $c=\mu_2$, and furthermore $u_{\mu_2}$ satisfies
\begin{align}
    \label{wentia}
\int_{\partial \Omega}u_{\mu_2}\, \rm{d} \sigma>0,    \end{align}
then $\kappa_1=\mu_2$. Moreover, in such situation, the functions at which the infimum in \eqref{defofkappa} is achieved must be linear combinations of Neumann Laplacian eigenfunctions of $\mu_2$.
\end{proposition}
We also show that the converse of the first statement in Proposition \ref{byproduct1} is not true. That is, even if there exists a solution to \eqref{cequation} when $c=\mu_2$, or equivalently \eqref{omegaproperty} holds for $k=2$, it might happen that $\kappa_1<\mu_2$. This is discussed in section 9 by studying comparison of $\kappa_1$ and $\mu_2$ on circular sectors. In fact, we show that there exist circular sectors on which $\mu_2$ has only odd modes, but it can happen on such sectors that $\kappa_1<\mu_2$, see Proposition \ref{disprovesector}. On the other hand, our numerical results on \eqref{cequation} in section 8 suggest that for isosceles triangles, $\kappa_1=\mu_2$ if and only if \eqref{omegaproperty} holds for $k=2$. This phenomenon deserves exploration in future work.

We remark that $\kappa_1$ is referred to as the first nonzero eigenvalue of boundary mean zero Laplacian with constant Neumann data, and thus we are led to study higher eigenvalues. Let $0<\kappa_1\le \kappa_2\le \cdots \le \kappa_i \rightarrow \infty$ be the eigenvalues of boundary mean zero Laplacian with constant Neumann data, and then the corresponding eigenfunctions $\{w_i\}_{i=1}^{\infty}$ satisfy
\begin{align}
    \label{eigenk}
\begin{cases}
-\Delta w_i=\kappa_i w_i\quad &\mbox{in $\Omega$}\\
\frac{\partial w_i}{\partial \nu}=-\frac{\kappa_i}{P(\Omega)}\int_{\Omega}w_i\,dx\quad &\mbox{on $\partial \Omega$}\\
\int_{\partial \Omega}w_i\, \rm{d} \sigma=0.
\end{cases}    
\end{align}
In general, we can prove that $\kappa_i \le \mu_{i+1}$. Moreover, by relating to the equation \eqref{cequation}, we obtain the following sufficient condition for $\kappa_i=\mu_{i+1}$, which is often useful in determining the value of $\kappa_i$, especially for the case $i=1$.
\begin{proposition}
\label{gensufintro}
Let $u_c$ denote the solution to \eqref{cequation}. Fixing a positive integer $i$, if there exists some $c \in [\kappa_i,\mu_{i+1}]$ such that $\int_{\partial \Omega}u_c\, \rm{d} \sigma>0$, then $\kappa_i=\mu_{i+1}$.
\end{proposition}

As a direct consequence, we prove that $\kappa_1=\mu_2$ when $\Omega$ is a ball, a rectangular box or an equilateral triangle. This extends \cite[Theorem 2.1 and Theorem 2.5]{NP15} and is obtained by much simpler proofs. Another interesting consequence is the following, due to Szeg\"o-Weinberger inequality (see \cite{Szego} and \cite{Wein}), its quantitative version (see \cite{BP12}) and our results.
\begin{theorem}
\label{iso}
Let $\Omega$ be a bounded Lipschitz domain in $\mathbb{R}^n$. Then $\kappa_1(\Omega)\le \kappa_1(\Omega^\sharp)$, where $\Omega^\sharp$ is the ball with the same volume as that of $\Omega$. Moreover,
\begin{align}
    \label{qk1}
|B|^{2/n}\kappa_1(B)-|\Omega|^{2/n}\kappa_1(\Omega)\ge C(n)A^2(\Omega),    
\end{align}
where $B$ is a ball and  $A(\Omega)$ is the Fraenkel asymmetry of $\Omega$.
\end{theorem}

Thanks to \cite{LS09} which considers maximizing $\mu_2$ on triangles, it is readily seen that among triangles with prescribed perimeter, equilateral triangles maximize $\kappa_1$, see section 6. 

Note that by Theorem \ref{maintheorem}, for the domain $\Omega$ on which $\kappa_1<\mu_2$, the infimum of $c>0$ such that $\int_{\partial \Omega}u_c\, \rm{d} \sigma<0$ is equal to $\kappa_1$. In the case of $\kappa_1=\mu_2$, we also have some partial results on the infimum of such number $c$, see Proposition \ref{c0} in section 5.

Finally, we remark that the question of fully classifying domains on which $\kappa_1=\mu_2$ still remains open. Even if Theorem \ref{maintheorem} gives an if and only if condition, it is still not completely satisfactory, as it does not give concrete geometric features of such domains. Nevertheless, the usefulness of Theorem \ref{maintheorem} lies in three aspects. First, it can be used to theoretically determine the value of $\kappa_1$ on certain domains beyond rectangular boxes. Second, it leads to the fact that $\kappa_1=\mu_2$ can be served as a necessary condition for domains on which \eqref{zheng2} holds. Third, directly solving the value $\kappa_1$ by numerical methods is not so easy, while numerically solving \eqref{cequation} is standard, and thus Theorem \ref{maintheorem} can be used by numerical analysis to determine whether or not $\kappa_1=\mu_2$ on a given domain.

\subsection{Concentration breaking on a thermal insulation problem}

As an application of the results we obtained above, we can extend the symmetry breaking results on a thermal insulation problem discovered by Bucur, Buttazzo and Nitsch in \cite{BBN}. Let us first recall their main result.
\begingroup
\renewcommand{\thetheorem}{\Alph{theorem}}
\begin{theorem}\emph{(\cite[Theorem 3.1]{BBN})}
\label{gongming}
Let $\Omega$ be a ball, $m>0$ and $u_m$ be a solution to 
\begin{align}
\label{2problem}
    \min\left\{\frac{\int_{\Omega}|\nabla u|^2\,dx+\frac{1}{m}\left(\int_{\partial \Omega}|u|\,\rm{d} \sigma\right)^2}{\int_{\Omega}u^2\,dx}:u\in H^1(\Omega)\setminus \{0\}\right\},
\end{align}
then there exists $m_0>0$ such that when $m>m_0$, $u_m$ is radial, while when $m<m_0$, $u_m$ is not radial.
\end{theorem}
\endgroup

For the mathematical modeling and some other related mathematical results, we refer to \cite{BBN1}, \cite{Buttazzo}  \cite{HLL}, \cite{HLL2}, \cite{PNST}, \cite{PNT}, etc. Physically, \eqref{2problem} is for minimizing temperature decay, and Theorem \ref{gongming} roughly says that for a round thermal body $\Omega$, when the total amount of material is larger than some threshold $m_0$, the insulating material should uniformly cover $\partial \Omega$, while below $m_0$, it is better to design a non-uniform cover. 

We interpret Theorem \ref{gongming} as a concentration breaking result, since essentially the authors show that when $m>m_0$, $u_m$ must be nonvanishing everywhere on the boundary, which is equivalent to its radial symmetry, and when $m<m_0$, $u_m$ must vanish on a subset of $\partial \Omega$ with positive $\mathscr{H}^{n-1}$ measure, and thus is not radial. The latter fact is due to perturbation by second Neumann Laplacian eigenfunctions, while the non-vanishing of $u_m$ on $\partial \Omega$ when $m>m_0$ is due to spherical symmetrization and thus highly relies on the assumption that $\Omega$ is a ball. 

When $\Omega$ is not a ball, by studying \eqref{cequation}, we extend the concentration breaking phenomenon from ball domains to wider class of symmetric domains including both ball domains and rectangular boxs. The class is given as follows:
\begin{align}
    \label{F}
\mathcal{F}=\{\mbox{$\Omega$ is a bounded Lipschitz domain in $\mathbb{R}^n$}: u_c>0 \, \mbox{on $\partial \Omega$ for any $c \in (0,\mu_2)$}\},
\end{align}where $u_c$ is the solution to (1.2).

On concentration breaking, so far our results can be summarized as below.
\begin{theorem}
\label{q2main}
Let $\Omega$ be a bounded Lipschitz domain and $u_m$ be a solution to \eqref{2problem}.
\begin{enumerate}
    \item There always exists $m_0>0$ such that when $m \in (0,m_0)$, $|u_m|=0$ on a subset of $\partial \Omega$ with positive $\mathscr{H}^{n-1}$ measure.
    \item If furthermore $\Omega \in \mathcal{F}$, then the above $m_0$ can be chosen such that when $m \in (m_0,\infty)$, $|u_m|>0$ on $\partial \Omega$.
    \item If $\Omega$ is convex or $C^1$, then there always exists $m_0'>0$ such that when $m \in (m_0',\infty)$, $|u_m|>0$ on $\partial \Omega$.
\end{enumerate}
\end{theorem}
Note that (1) and (3) in Theorem \ref{q2main} address arbitrary domains, and (2) deals specifically with domains in the class $\mathcal{F}$ which particularly includes balls and rectangular boxes. Numerical results in section 8 indicate that planar elliptic domains are also in this class, while any triangular domain is not. It would be very interesting to rigorously prove this numerical result and then give further common features of domains in $\mathcal{F}$.

Last, we remark that when $\Omega$ is a ball, the exact value of breaking threshold $m_0$ is obtained in \cite{HLL} via implicit and rather complicated argument.  As a consequence of Theorem \ref{zheng1} and Theorem \ref{zheng2}, we give quantitative estimates for the breaking threshold $m_0$ for balls and rectangular boxes in a more direct way, and it turns out that for such domains, $m_0$ has the sharp lower bound given by $\tfrac{(n-1)P^2(\Omega)}{n\mu_2(\Omega)|\Omega|}$, see Corollaries \ref{quan1}-\ref{quan2}. 

\medskip

\subsection*{Acknowledgment} We would like to thank Professor Dorin Bucur and Professor Antoine Henrot for very helpful comments. We would also like to express our gratitude to the anonymous reviewers for their numerous helpful comments, which have significantly improved the present paper.

\medskip

\subsection*{Outline of the Paper}

The organization of the paper is as follows. In section 2, we prove some preliminary  properties on eigenvalues of boundary mean zero Laplacian with constant Neumann data, and we also prove Theorem \ref{existence}. In section 3, we study \eqref{cequation} on balls and also prove Theorem \ref{zheng1}. In section 4, we study \eqref{cequation} on rectangular boxes and prove Theorem \ref{zheng2}. In section 5, we prove Theorem \ref{maintheorem}, Proposition \ref{byproduct1}, Proposition \ref{gensufintro}, Theorem \ref{iso} and some other related miscellaneous results. In section 6, we study \eqref{cequation} on triangles and also prove isoperimetric results. In section 7, we prove Theorem \ref{q2main} and also obtain quantitative results on breaking threshold when $\Omega$ is a ball or rectangular box. In section 8, we provide numerical results on the boundary behaviors of solutions to \eqref{cequation} on some other symmetric domains. In section 9, we give counterexamples to the converse of the first statement of Proposition \ref{byproduct1}.

\section{eigenvalue comparison and solvability of \eqref{cequation}}
In this section, we first prove comparison results on Neumann Laplacian eigenvalues and the eigenvalues of the boundary mean zero Laplacian with constant Neumann data.

We define 
\begin{align}
    \label{W}
W(\Omega):=\Big\{w \in H^1(\Omega): \int_{\partial \Omega}w\, \rm{d} \sigma=0\Big\}.
\end{align}
Then the variational characterization of the $i$-th nonzero eigenvalue of boundary mean zero Laplacian with constant Neumann data is given as
\begin{align}
    \label{kappa_i}
\kappa_i(\Omega):=\inf\Big\{\frac{\int_{\Omega}|\nabla w|^2 \, dx}{\int_{\Omega}w^2 \, dx }: w \in W(\Omega) \setminus\{0\},\, u \perp \{w_1,\cdots,w_{i-1}\}\Big\}.    
\end{align}
In the above $w_1,\cdots,w_{i-1}\in W(\Omega)$ are eigenfunctions of the first $(i-1)$ eigenvalues of the boundary mean zero Laplacian under constant Neumann data, and the perpendicularity is in the sense of $L^2$ scalar product. For each eigenvalue $\kappa_i$, it is easy to check by the variational principle that the corresponding eigenfunction $w_i$ satisfies 
\begin{align*}
    \begin{cases}
    -\Delta w_i=\kappa_i w_i\quad &\mbox{in $\Omega$}\\
    \frac{\partial w_i}{\partial \nu}=-\frac{\kappa_i}{P(\Omega)}\int_{\Omega}w_i\,dx \quad &\mbox{on $\partial \Omega$}\\
    \int_{\partial \Omega}w_i\, \rm{d} \sigma=0.
    \end{cases}
\end{align*}
Also, similar to the Dirichlet or Neumann eigenvalue problem, one can easily verify that the min-max principle also holds in this new eigenvalue problem. That is,
\begin{align}
    \label{minmaxforwproblem}
\kappa_i(\Omega)=\min_{L \subset W(\Omega),\rm{dim} L=i}\,\max_{u \in L\setminus \{0\}}\frac{\int_{\Omega}|\nabla u|^2 \, dx}{\int_{\Omega}u^2\,dx}.    
\end{align}
Let $0=\mu_1<\mu_2\le\mu_3\le \cdots\le \mu_i\rightarrow \infty$ be the Neumann Laplacian eigenvalues, then directly from the min-max principle, we have $\mu_i \le \kappa_i$. Hence $\kappa_i \rightarrow \infty$ as $i \rightarrow \infty$. Let $v_i$ be the Neumann-Laplacian eigenfunction corresponding to $\mu_i$, and let $L=span\{v_2,\cdots, v_{i+1}\}$. For each $2 \le k\le i+1$, let $\bar{v}_k=v_k-\frac{1}{P(\Omega)}\int_{\partial \Omega}v_k\,\rm{d} \sigma$ and $\bar{L}=span\{\bar{v}_2,\cdots, \bar{v}_{i+1}\}$. Since $\bar{v}_k \in W(\Omega)$ and they are linearly independent, $\bar{L}\subset W(\Omega)$ and $dim \bar{L}=i$. Since
	\begin{align*}
		\int_{\Omega}|\nabla (\sum_{l=2}^{i+1} a_{l}\bar{v}_{l})|^2 \, dx=\sum_{l=2}^{i+1}a_{l}^2\int_{\Omega}|\nabla  v_{l}|^2\, dx
		\le  \mu_{i+1}\sum_{l=2}^{i+1}a_{l}^2\int_{\Omega} v_{l}^2\,dx=\mu_{i+1}\int_{\Omega}(\sum_{l=2}^{i+1}a_{l}  v_{l})^2\, dx
	\end{align*}and
	\begin{align*}
		\int_{\Omega}(\sum_{l=2}^{i+1}a_{l} \bar{v}_{l})^2\, dx=\int_{\Omega}(\sum_{l=2}^{i+1}a_{l}  v_{l})^2\, dx+|\Omega|\left(\sum_{l=2}^{i+1} \frac{a_{l}}{P(\Omega)}\int_{\partial \Omega} v_{l}\,\rm{d} \sigma\right)^2,
	\end{align*}
we get that for any $u\in \bar{L}\setminus \{0\}$,
	\begin{align*}
		\frac{\int_{\Omega}|\nabla u|^2 \, dx}{\int_{\Omega}u^2\,dx} \le \mu_{i+1}.
	\end{align*}
	Hence by the min-max principle, we have that $\kappa_i \le \mu_{i+1}$.

In summary, we have proved:
\begin{proposition}
\label{gaojiebijiao}
Let $\kappa_1\le \kappa_2 \le \cdots \le \kappa_i \rightarrow \infty$ be the eigenvalues (counting multiplicity) of the boundary mean zero Laplacian with constant Neumann data. Let $0=\mu_1<\mu_2\le\mu_3\le \cdots\le \mu_i\rightarrow \infty$ be the Neumann Laplacian eigenvalues. Then for each $i \ge 1$,  $\mu_i \le \kappa_i \le \mu_{i+1}$.
\end{proposition}

The following corollary is immediate from the above proposition.
\begin{corollary}
The eigenvalues $\kappa_i$ satisfy the same Weyl's law as for Neumann Laplacian eigenvalues.
\end{corollary}


Before proving Theorem \ref{existence}, we shall need the following Fredholm alternative for Neumann conditions. 
\begin{theorem}[\rm{Fredholm alternative}]
\label{Fred}
Given $\Omega$ to be a bounded Lipschitz domain in $\mathbb{R}^n$, $c>0$, and $f \in L^2(\Omega)$. Consider the nonhomogeneous eigenvalue equation:
\begin{align}
    \label{fredequation}
\begin{cases}
-\Delta u=cu+f\quad &\mbox{in $\Omega$}\\
\frac{\partial u}{\partial \nu}=0\quad &\mbox{on $\partial \Omega$}.
\end{cases}    
\end{align}
Then when $c$ does not belong to any Neumann Laplacian eigenvalues, \eqref{fredequation} admits a unique weak solution in $H^1(\Omega)$. When $c$ is a Neumann Laplacian eigenvalue, then \eqref{fredequation} admits a weak solution in $H^1(\Omega)$ if and only if $f$ is perpendicular (for the $L^2$ scalar product) to any Neumann Laplacian eigenfunction associated to $c$. 
\end{theorem}
The statement of Theorem \ref{Fred} is similar to the Fredholm alternative for Dirichlet boundary conditions, which is well known and can be found for example in \cite{CH} and \cite{Henrot2}. We have not found a reference stating Fredholm alternative for Neumann conditions, but standard argument as in \cite{Evans} can prove Theorem \ref{Fred}, by working on the Hilbert space $\tilde{H}(\Omega)$ which is the set of functions in $H^1(\Omega)$ with vanishing mean over $\Omega$, and by exact same argument in \cite[Section 6.2]{Evans}, one can consider a weak solution $u \in \tilde{H}(\Omega)$ satisfying the corresponding integral equation for any test function in $\tilde{H}(\Omega)$, and by adding a constant to $u$, one can show that the integral equation is true for all test functions in $H^1(\Omega)$, and thus we find a weak solution in $H^1(\Omega)$. Here we omit the details of the proof since they are standard.

Throughout the following, by saying a solution to \eqref{cequation}, we mean weak solution.

Now let us prove Theorem \ref{existence}.
\begin{proof}[Proof of Theorem \ref{existence}]
Let $c=\mu_k$, and we first show that \eqref{omegaproperty} is a necessary condition for the existence of solution to \eqref{cequation}. Indeed, let $v$ be a Neumann Laplacian eigenfunction of $\mu_k$, then from the second Green identity
\begin{align*}
    \int_{\Omega}(\Delta u_c v-\Delta v u_c)\, dx=\int_{\partial \Omega} \left(\frac{\partial u_c}{\partial \nu} v-\frac{\partial v}{\partial \nu}u_c\right)\, \rm{d} \sigma,
\end{align*}we derive \eqref{omegaproperty}.

The sufficiency of \eqref{omegaproperty} is a consequence of Fredholm alternative. The Rieze representation theorem gives the existence of $f \in H^1(\Omega)$ such that 
\begin{align*}
    \begin{cases}
    -\Delta f=\frac{P(\Omega)}{|\Omega|}\quad &\mbox{in $\Omega$}\\
    \frac{\partial f}{\partial \nu}=-1 \quad &\mbox{on $\partial \Omega$}.
    \end{cases}
\end{align*}
Hence it suffices to prove that for such $f$, the equation
\begin{align*}
    \begin{cases}
    -\Delta u=\mu_k u+\mu_kf-\frac{P(\Omega)}{|\Omega|}\quad &\mbox{in $\Omega$}\\
    \frac{\partial u}{\partial \nu}=0 \quad &\mbox{on $\partial \Omega$}
    \end{cases}
\end{align*}admits a solution. By Theorem \ref{Fred}, it suffices to verify $\mu_k f-\frac{P(\Omega)}{|\Omega|}$ is perpendicular to any Neumann Laplacian eigenfunction $v$ of $\mu_k$, and it suffices to show $f \perp v$, which is again a consequence of second Green's identity and $\eqref{omegaproperty}$.
\end{proof}

Thanks to Theorem \ref{existence}, and due to Dirichlet and Neumann eigenvalue comparison results and the seminal work of Laugesen-Siudeja \cite{LS10}, the next corollary gives some common symmetric domains on which \eqref{cequation} admits a solution when $c=\mu_2$.

\begin{corollary}
Let $c=\mu_2$. Then \eqref{cequation} is solvable when $\Omega$ is an isosceles triangle with aperture greater than or equal to $\frac{\pi}{3}$, or $\Omega$ is a convex domain with two axes of symmetry.
\end{corollary}
\begin{proof}
By \cite{LS10}, on isosceles triangles with aperture greater than $\frac{\pi}{3}$, any eigenfunction of $\mu_2$ is antisymmetric with respect to the bisection line, and hence \eqref{omegaproperty} is satisfied for $k=2$. For equilateral triangular case, \eqref{omegaproperty} is also satisfied, see \eqref{equisatisfy}.

If $\Omega$ is a convex domain with two axes of symmetry, then by the fact that $\mu_2<\lambda_D$ for any domain, where $\lambda_D$ is the first eigenvalue of Dirichlet Laplacian, the nodal line of any Neumann Laplacian eigenfunction of $\mu_2$ must touch the boundary. Also, by the Courant nodal domain theorem (see \cite{CH}), the nodal line divides the domain into two parts. Hence any eigenfunction cannot be even, and thus the Neumann Laplacian eigenspace of $\mu_2$ is spanned by odd functions. This implies \eqref{omegaproperty}.
\end{proof}

\section{The equation \eqref{cequation} on balls and proof of Theorem \ref{zheng1}}
In this section, we will prove the two conclusions in Theorem \ref{zheng1}.

Even though \eqref{quyu0} is true for any Lipschitz domain, we would like to first prove \eqref{quyu0} and \eqref{ballcase} together for ball domains, since both of the proofs use properties of Bessel functions. We first recall that Bessel function of first kind of order $s$ is denoted as $J_s(z)$. The following three formulas are well known, see for example \cite{K02}.
\begin{align}
\label{1}
    J_s'(z)=-J_{s+1}(z)+\frac{sJ_s(z)}{z},
\end{align} 
\begin{align}
\label{2}
  J_s'(z)=J_{s-1}(z)-\frac{sJ_s(z)}{z},
\end{align}
and
\begin{align}
    \label{3}
\lim_{z \rightarrow 0}\frac{zJ_s'(z)}{J_s(z)}=s.    
\end{align}

Based on these properties, we have the following result.
\begin{proposition}
\label{ballmotivation}
Let $\Omega$ {be the ball of radius $R$ centered at the origin}, $c\in (0,\mu_2)$ and $u_c$ be a solution to \eqref{cequation}. Then $u_c$ satisfies both \eqref{quyu0} and \eqref{ballcase}.
\end{proposition}

\begin{proof}
If $c<\mu_2$, then \eqref{cequation} admits a unique solution which must be radial, and hence we may set $u=a_cr^{1-\frac{n}{2}}J_{\frac{n}{2}-1}(\sqrt{c}r)$. Using $\eqref{cequation}_2$, we have
\begin{align*}
    a_c(1-\frac{n}{2})R^{-\frac{n}{2}}J_{\frac{n}{2}-1}(\sqrt{c}R)+a_cR^{1-\frac{n}{2}}\sqrt{c}J_{\frac{n}{2}-1}'(\sqrt{c}R)=-1.
\end{align*}Hence from the above and \eqref{1}-\eqref{2}, we have
\begin{align*}
    c\int_{\partial \Omega}u_c\, \rm{d} \sigma=&-\frac{cn\omega_nR^nJ_{\frac{n}{2}-1}(\sqrt{c}R)}{(1-\frac{n}{2})J_{\frac{n}{2}-1}(\sqrt{c}R)+\sqrt{c}RJ_{\frac{n}{2}-1}'(\sqrt{c}R)}\\
    =&-\frac{cn\omega_nR^nJ_{\frac{n}{2}-1}(\sqrt{c}R)}{(1-\frac{n}{2})J_{\frac{n}{2}-1}(\sqrt{c}R)+\sqrt{c}R\left(-J_{\frac{n}{2}}(\sqrt{c}R)+\frac{(\frac{n}{2}-1)J_{\frac{n}{2}-1}(\sqrt{c}R)}{\sqrt{c}R}\right)}\\
    =&-\frac{cn\omega_nR^nJ_{\frac{n}{2}-1}(\sqrt{c}R)}{\sqrt{c}R\left(\frac{2\sqrt{c}R}{n}\left(J_{\frac{n}{2}}'(\sqrt{c}R)-J_{\frac{n}{2}-1}(\sqrt{c}R)\right)\right)}\\
    =&\frac{1}{2}\frac{n^2\omega_nR^{n-2}J_{\frac{n}{2}-1}(\sqrt{c}R)}{J_{\frac{n}{2}-1}(\sqrt{c}R)-J_{\frac{n}{2}}'(\sqrt{c}R)}.
\end{align*}
In the following, we use $B_R$ to denote the ball of radius $R$ centered at the origin. Using the fact that $R\sqrt{\mu_2}=\sqrt{\mu_2(B_1)}$, which is the first positive root of the following equation (see for example \cite[Page 233]{nshap}) 
\begin{align}
\label{xiuzheng}
    zJ_{\frac{n}{2}}'(z)-\frac{n-2}{2}J_{\frac{n}{2}}(z)=0,
\end{align}
we have
\begin{align}
\label{yao}
    \lim_{c<\mu_2(B_R),c \rightarrow \mu_2(B_R)}c\int_{\partial \Omega}u_c\, \rm{d} \sigma=\frac{1}{2}\frac{P^2(B_R)}{|B_R|}\frac{J_{\frac{n}{2}-1}(\mu_2(B_1))}{J_{\frac{n}{2}-1}(\mu_2(B_1))-\frac{(n-2)J_{\frac{n}{2}}(\mu_2(B_1))}{2\mu_2(B_1)}}.
\end{align}
By \eqref{2}, \eqref{xiuzheng} becomes
\begin{align}
    \label{yao1}
zJ_{\frac{n}{2}-1}(z)=(n-1)J_{\frac{n}{2}}(z).  
\end{align}
Hence by \eqref{yao}-\eqref{yao1}, we have
\begin{align*}
    \lim_{c<\mu_2(B_R),c \rightarrow \mu_2(B_R)}c\int_{\partial \Omega}u_c\, \rm{d} \sigma=\frac{1}{2}\frac{P^2(B_R)}{|B_R|}\frac{1}{\left(1-\frac{n-2}{2(n-1)}\right)}=\frac{n-1}{n}\frac{P^2(B_R)}{|B_R|}.
\end{align*}
This proves \eqref{ballcase}.

To prove \eqref{quyu0}, by \eqref{2}, we also have
\begin{align*}
    c\int_{\partial \Omega}u_c\, \rm{d} \sigma =\frac{n^2\omega_nR^{n-2}}{2}\frac{zJ_{s-1}(z)}{sJ_s(z)},
\end{align*}where $z=\sqrt{c}R$ and $s=\frac{n}{2}$. Since as $c\rightarrow 0$, $z \rightarrow 0$, and thus by \eqref{2} and \eqref{3}, we have
\begin{align*}
   \frac{zJ_{s-1}(z)}{sJ_s(z)}= \frac{zJ_s'(z)+sJ_s(z)}{sJ_s(z)}\rightarrow 2, \quad \mbox{as $c\rightarrow 0$.}
\end{align*}Hence
\begin{align*}
    \lim_{c\rightarrow 0}c\int_{\partial \Omega}u_c\, \rm{d} \sigma=n^2\omega_nR^{n-2}=\frac{P^2(B_R)}{|B_R|}.\end{align*}
\end{proof}

Next, we justify \eqref{quyu0} on any Lipschitz domain. We actually have the following stronger result.
\begin{theorem}
\label{main2}
Let $\Omega$ be a bounded Lipscthiz domain in $\mathbb{R}^n$. Then
\begin{align}
    \label{fcomega}
f(c,\Omega):=c\int_{\partial \Omega}u_c\, \rm{d} \sigma
\end{align}is a strictly decreasing positive function of $c$ when $c \in (0,\kappa_1)$, where $u_c$ is the solution to \eqref{cequation}, and $\kappa_1$ is the first nonzero eigenvalue of boundary mean zero with constant Neumann data, defined in \eqref{defofkappa}. Moreover,
\begin{align}
    \label{0omega}
\lim_{c\rightarrow 0}f(c,\Omega)=\frac{P^2(\Omega)}{|\Omega|}.    
\end{align}
\end{theorem}
\begin{proof}
We define the following auxiliary functional 
\begin{align}
    \label{kappam}
\kappa(m,\Omega):=\inf\Big\{\frac{\int_{\Omega}|\nabla u|^2 \, dx+\frac{1}{m}\left(\int_{\partial \Omega}u\,\rm{d} \sigma\right)^2}{\int_{\Omega}u^2\, \rm{d} \sigma}: u\in H^1(\Omega)\setminus \{0\}\Big\},
\end{align}and we let $u_m$ denote the function at which the infimum in \eqref{kappam} is achieved, with
\begin{align*}
    \int_{\Omega}u_m^2\, dx=1.
\end{align*}
Since $\kappa(m,\Omega)\rightarrow 0$ as $m \rightarrow \infty$, we set 
\begin{align}
    \label{m0}
m_0:=\inf\{m> 0:\kappa(m,\Omega)<\kappa_1 \}. \end{align}
Standard argument can show that $\kappa(m,\Omega)$ is continuous with respect to $m>0$. Also, when $m>m_0$, $\int_{\partial \Omega}u_m\,\rm{d} \sigma\ne 0$, and thus $\kappa(m,\Omega)$ is a strictly decreasing function for $m\in (m_0,\infty)$. Therefore, if $c\in (0,\kappa_1)$, there exists a unique number $m_c\in(m_0,\infty)$ such that $c=\kappa(m_c, \Omega)$, and $m_c$ is strictly increasing as $c$ is strictly decreasing. Additionally, as $c \rightarrow 0$, it forces that $u_{m_c}$ converges to a nonzero constant weakly in $H^1(\Omega)$, which further implies that
\begin{align}
    \label{limitmc}
\lim_{c\rightarrow 0}m_c=\infty.    
\end{align}
From minimality, $u_{m_c}$ satisfies
\begin{align}
    \label{umc'}
\begin{cases}
-\Delta u_{m_c}=cu_{m_c}\quad &\mbox{in $\Omega$}\\
\frac{\partial u_{m_c}}{\partial \nu}=-\frac{1}{m_c}\int_{\partial \Omega}u_{m_c}\, \rm{d} \sigma\quad &\mbox{on $\partial \Omega$.}
\end{cases}
\end{align}
Let 
\begin{align}
    \label{lianxiuc}
u_c=\frac{m_cu_{m_c}}{\int_{\partial \Omega}u_{m_c}\,\rm{d} \sigma},    
\end{align}
and thus $u_c$ is a solution to \eqref{cequation}. It is unique since $c<\kappa_1 \le \mu_2$, due to Proposition \ref{gaojiebijiao}. Hence if $c<\kappa_1$, we can write
\begin{align}
\label{rewritee}
    f(c,\Omega):=c\int_{\partial \Omega}u_c\, \rm{d} \sigma=m_c\kappa(m_c,\Omega)=\inf\Big\{\frac{m_c\int_{\Omega}|\nabla u|^2 \, dx+\left(\int_{\partial \Omega}u\,\rm{d} \sigma\right)^2}{\int_{\Omega}u^2\, \rm{d} \sigma}: u\in H^1(\Omega)\Big\},
\end{align}which is clearly positive and decreasing as $c\in (0,\kappa_1)$ is increasing. $f(c,\Omega)$ is actually strictly decreasing, because the Euler-Lagrange equation to the above minimization problem implies that the solution is not a constant. 

Finally, by \eqref{limitmc} and the formula \eqref{rewritee}, we obtain \eqref{0omega}.
\end{proof}

At the end, we give a remark which can be seen from the proof above.
\begin{remark}
\label{strictlypositive}
When $n=2, 3$ and $\Omega$ is a bounded Lipschitz domain, there exists $c_0>0$ such that $u_c>0$ on $\partial \Omega$ when $c \in (0,c_0)$. The same conclusion holds in any dimension if $\Omega$ is $C^1$ or convex.. 
\end{remark}
\begin{proof}
Let $u_m$ and $\kappa(m,\Omega)$ be as in those in the proof of Theorem \ref{main2}. As $m \rightarrow \infty$, it is clear that $u_m$ converges to a nonzero constant weakly in $H^1(\Omega)$. In fact, by H\"older estimates for harmonic functions with bounded Neumann boundary conditions, see \cite{Kenig} and \cite{KP93}, together with elliptic estimates, $u_m \in C^\alpha(\bar{\Omega})$ for some $\alpha \in (0,1)$ when $n=2,3$. If $\Omega$ is $C^1$ or convex, then according to \cite{FMM98} and \cite{GS10}, we still have the H\"older estimate in any dimension. Hence $u_m$ converges to the nonzero constant uniformly when $m \rightarrow \infty$ in these situation. Hence $u_m$ does not vanish on $\partial \Omega$ when $m> M$, for some large $M>0$. We may assume that when $m>M$, $\kappa(m,\Omega)<\kappa_1$ and thus for any $0<c<c_0:=\kappa(M,\Omega)$, there is $m_c>M$ such that $u_c=\tfrac{m_cu_{m_c}}{\int_{\partial \Omega}u_{m_c}\, \rm{d} \sigma}$ is the unique solution to \eqref{cequation}, and hence $u_c>0$ on $\partial \Omega$ when $c \in (0,c_0)$.
\end{proof}

\section{Proof of Theorem \ref{zheng2}}
Before proving Theorem \ref{zheng2}, we first remark that when $c=\mu_2$, even if there are infinitely many solutions to \eqref{cequation}, these solutions are the same up to adding a Neumann Laplacian eigenfunction of $\mu_2$. Since all Neumann Laplacian eigenfunctions of $\mu_2$ on rectangular boxes are spanned by anti-symmetric functions, they have boundary mean zero, and hence $\int_{\partial \Omega}u_{\mu_2}\,\rm{d} \sigma$ is well-defined.

\begin{proof}[Proof of Theorem \ref{zheng2}]
We may assume that $\Omega=\Pi_{i=1}^n(-a_i,a_i)$ and $a_1 \ge a_2 \ge \cdots \ge a_n$. Let $\sigma_k$ be the $k$-th elementary symmetric polynomial with respect to $a_1, \cdots, a_n$.  Hence 
\begin{align}
\label{guanxi}
    |\Omega|=2^n \sigma_n\quad \mbox{and} \quad P(\Omega)=2^n\sigma_{n-1}.
\end{align}
It is easily constructed that 
\begin{align}
    \label{recsolution}
u_c=\frac{1}{\sqrt{c}}\sum_{i=1}^n \frac{\cos(\sqrt{c}x_i)}{\sin(\sqrt{c}a_i)}
\end{align}is a solution to \eqref{cequation}. Hence
\begin{align}
\label{jisuan1}
    \int_{\partial \Omega}u_c\,\rm{d} \sigma=&2\sum_{i=1}^n\int_{\{x_i=a_i\}}u_c \, \rm{d} \sigma\nonumber\\
    =&\frac{2}{\sqrt{c}}\sum_{i=1}^n\left(\frac{\cos(\sqrt{c}a_i)}{\sin(\sqrt{c}a_i)}\frac{|\Omega|}{2a_i}+\sum_{j\ne i}\frac{|\Omega|}{2a_i\cdot 2a_j}\int_{-a_j}^{a_j}\frac{\cos(\sqrt{c}x_j)}{\sin(\sqrt{c}a_j)}\, dx_j\right)\nonumber\\
    =&\sum_{i=1}^n\left(\frac{1}{\sqrt{c}}\frac{|\Omega|}{a_i}\frac{\cos(\sqrt{c}a_i)}{\sin(\sqrt{c}a_i)}+\frac{1}{c}\frac{|\Omega|}{a_i}\sum_{j\ne i}\frac{1}{a_j}\right).
\end{align}
Since $\mu_2=\left(\frac{\pi}{2a_1}\right)^2$, we have that 
\begin{align}
\label{left1}
    \lim_{c\rightarrow \mu_2} c\int_{\partial \Omega}u_c\,\rm{d} \sigma=\sum_{i=2}^n2^{n-1}\pi \frac{\sigma_n}{a_1a_i}\frac{1}{\tan (\frac{\pi}{2}\frac{a_i}{a_1})}+2^{n+1}\sigma_{n-2}.
\end{align}
Using that 
$$\frac{2}{\pi}\theta \le \sin \theta \le \theta,\quad \forall \theta \in [0,\frac{\pi}{2}],$$
we have that 
\begin{align}
    \label{tan}
\frac{1}{\tan \theta}=\frac{\sin(\frac{\pi}{2}-\theta)}{\sin\theta}\ge \frac{\frac{2}{\pi}(\frac{\pi}{2}-\theta)}{\theta},\quad \forall \theta \in (0,\frac{\pi}{2}].
\end{align}
By \eqref{tan}, and since $\frac{\pi}{2}\frac{a_i}{a_1}\in (0,\frac{\pi}{2}]$, \eqref{left1} can be estimated as 
\begin{align}
    \label{left2}
\lim_{c\rightarrow \mu_2} c\int_{\partial \Omega}u_c\,\rm{d} \sigma \ge \sum_{i=2}^n 2^n\frac{\sigma_n}{a_i^2}(1-\frac{a_i}{a_1})+2^{n+1}\sigma_{n-2}.
\end{align}
Let $d_k$ be the $k$-th elementary symmetric polynomials in the $n-1$ variables $(a_2,\dots,a_n)\,$. It follows from the definitions that $\,\sigma_n=a_1d_{n-1}\,$ and $\,\sigma_k=d_k+a_1d_{k-1}\,$ for $\,1 \le k \le n-1\,$.
Using the identities
\begin{align*}
    \sum_{i=2}^n \frac{1}{a_i} = \frac{d_{n-2}}{d_{n-1}}
\end{align*} and
\begin{align*}
    \sum_{i=2}^n \frac{1}{a_i^2} = \left(\frac{d_{n-2}}{d_{n-1}}\right)^2-2 \frac{d_{n-3}}{d_{n-1}},
\end{align*}from \eqref{left2} we obtain that
\begin{align}
    \label{left3}
\lim_{c\rightarrow \mu_2} c\int_{\partial \Omega}u_c\,\rm{d} \sigma \ge&2^n a_1\left(\frac{d_{n-2}^2}{d_{n-1}}-2d_{n-3}\right)-2^nd_{n-2} +2^{n+1}(d_{n-2}+a_1d_{n-3})\nonumber\\
=&2^n\frac{d_{n-2}}{d_{n-1}}(d_{n-1}+a_1d_{n-2})\nonumber\\
=&2^n\frac{a_1d_{n-2}}{\sigma_n}\sigma_{n-1}.
\end{align}
Also,
\begin{align*}
    \frac{P^2(\Omega)}{|\Omega|}=2^n\frac{\sigma_{n-1}^2}{\sigma_n}.
\end{align*}
Hence
\begin{align}
\label{jian}
\lim_{c\rightarrow \mu_2} c\int_{\partial \Omega}u_c\,\rm{d} \sigma-\frac{n-1}{n}\frac{P^2(\Omega)}{|\Omega|}  \ge& \frac{2^n\sigma_{n-1}}{n\sigma_n}\left(na_1d_{n-2}-(n-1)\sigma_{n-1}\right)\nonumber \\
=&\frac{2^n\sigma_{n-1}}{n\sigma_n}\left(na_1d_{n-2}-(n-1)(d_{n-1}+a_1d_{n-2})\right)\nonumber\\
=&\frac{2^n\sigma_{n-1}}{n\sigma_n}(a_1d_{n-2}-(n-1)d_{n-1}).
\end{align}
Let $e_k$ be the $k$-th elementary polynomial in the $(n-1)$ variables $\frac{a_2}{a_1},\cdots, \frac{a_n}{a_1}$. Hence \eqref{jian} is equivalent to
\begin{align}
    \label{jian2}
 \lim_{c\rightarrow \mu_2} c\int_{\partial \Omega}u_c\,\rm{d} \sigma-\frac{n-1}{n}\frac{P^2(\Omega)}{|\Omega|}  \ge  \frac{2^n\sigma_{n-1}}{n\sigma_n} a_1^{n-1}(e_{n-2}-(n-1)e_{n-1}).
\end{align}
Applying the Maclaurin's inequalities to $e_k$, we particularly have
\begin{align}
    \label{ekmac}
\left(\frac{e_{n-2}}{n-1}\right)^{\frac{1}{n-2}}\ge e_{n-1}^{\frac{1}{n-1}}.    
\end{align}
Hence by \eqref{jian2} and \eqref{ekmac},
\begin{align}
    \label{jian3}
 \lim_{c\rightarrow \mu_2} c\int_{\partial \Omega}u_c\,\rm{d} \sigma-\frac{n-1}{n}\frac{P^2(\Omega)}{|\Omega|}  \ge  \frac{2^n\sigma_{n-1}}{n\sigma_n} a_1^{n-1}(n-1)(e_{n-1}^{\frac{n-2}{n-1}}-e_{n-1})\ge 0,
\end{align}where the last inequality is due to the fact that $\frac{a_i}{a_1} \le 1$, for $\, i=2,\dots, n$. This finishes the proof.
\end{proof}

As a corollary, among rectangular boxes with prescribed perimeter, cubes are the unique minimizers to $\mu_2\int_{\partial \Omega}u_{\mu_2}\, \rm{d} \sigma$.
\begin{corollary}
Let $c \in (0,\mu_2)$ and $u_c$ be the solution to \eqref{cequation}. Then among the class of rectangular boxes $\Omega \subset \mathbb{R}^n$, we have
    \begin{align}
\label{recin0}
    \lim_{c\rightarrow \mu_2}c\int_{\partial \Omega}u_c\, \rm{d} \sigma=\mu_2\int_{\partial \Omega}u_{\mu_2}\, \rm{d} \sigma \ge 2^nn(n-1)\left(\frac{P(\Omega)}{n2^n}\right)^{\frac{n-2}{n-1}}.
\end{align}
\end{corollary}

\begin{proof}
Let $\sigma_k$ be as before. Then \eqref{recin0} is an immediate consequence of \eqref{jianrenzhichang} and the following relation
\begin{align*}
    |\Omega|=2^n\sigma_n \le 2^n \left(\frac{\sigma_{n-1}}{n}\right)^{\frac{n}{n-1}}=2^n\left(\frac{P(\Omega)}{n2^n}\right)^{\frac{n}{n-1}},
\end{align*}due to Maclaurin's inequality.
\end{proof}

\section{Necessary and Sufficient conditions for $\kappa_1=\mu_2$}
In this section, the main goal is to establish necessary and sufficient conditions for $\Omega$ on which $\kappa_1=\mu_2$, by relating to the equation \eqref{cequation}. 

Recall that $0<\kappa_1\le \kappa_2\le \cdots \le \kappa_i \rightarrow \infty$ are the eigenvalues of boundary mean zero Laplacian with contant Neumann data, and then the corresponding eigenfunctions $\{w_i\}_{i=1}^{\infty}$ satisfy \eqref{eigenk}. We first prove Proposition \ref{gensufintro}, which gives a sufficient condition on $\kappa_i=\mu_{i+1}$. Our proof is inspired by \cite{CH} and \cite{Friedlander}.

\begin{proof}[Proof of Proposition \ref{gensufintro}]
Suppose that $\kappa_i<\mu_{i+1}$, and we will prove by contradiction. Let $w_i$ be a boundary mean zero Laplacian eigenfunction of $\kappa_i$ under constant Neumann condition. Let $L:=span\{u_c, w_1, \cdots, w_i\}$, where $c$ and $u_c$ are as in hypothesis. Hence $dim L=i+1$. For any $u=au_c+\sum_{l=1}^i b_l w_l \ne 0$, with $a\in \mathbb{R}$, $b_l\in \mathbb{R}$, we have
\begin{align}
\label{lalian1}
    \int_{\Omega}|\nabla u|^2\, dx=&a^2 \int_{\Omega}|\nabla u_c|^2\, dx+\sum_{l=1}^ib_l^2 \int_{\Omega}|\nabla w_l|^2\, dx+2\sum_{l=1}^iab_l\int_{\Omega}\nabla u_c \nabla w_l\, dx\nonumber\\
    =&a^2 \int_{\Omega}|\nabla u_c|^2\, dx+\sum_{l=1}^ib_l^2 \int_{\Omega}|\nabla w_l|^2\, dx+2c\sum_{l=1}^iab_l\int_{\Omega} u_c w_l\, dx.
\end{align}
If $a=0$, then by \eqref{lalian1},
\begin{align}
\label{lalian2}
    \int_{\Omega}|\nabla u|^2 \, dx=\sum_{l=1}^ib_l^2 \int_{\Omega}|\nabla w_l|^2\, dx=\sum_{l=1}^i \kappa_lb_l^2 \int_{\Omega}w_l^2\, dx<\mu_{i+1}\int_{\Omega}u^2 \, dx.
\end{align}
If $a \ne 0$, then by \eqref{cequation} and \eqref{lalian1}, we have
\begin{align}
\label{lalian3}
    \int_{\Omega}|\nabla u|^2 \, dx<& a^2 c\int_{\Omega}u_c^2 \, dx+\sum_{l=1}^i \kappa_l b_l^2 \int_{\Omega}w_l^2 \, dx+2c\sum_{l=1}^iab_l\int_{\Omega} u_c w_l\, dx\nonumber\\
    \le &c\int_{\Omega}(au_c+\sum_{l=1}^i b_l w_l)^2 \, dx\nonumber\\
    \le & \mu_{i+1}\int_{\Omega}u^2 \, dx.
\end{align}
Therefore, by \eqref{lalian2} and \eqref{lalian3}, we obtain
\begin{align*}
    \inf_{L \subset H^1(\Omega), \rm{dim} L=i+1}\max_{u\in L\setminus \{0\}}\frac{\int_{\Omega}|\nabla u|^2 \, dx}{\int_{\Omega}u^2 \, dx}<\mu_{i+1},
\end{align*}which contradicts to min-max principle for Neumann Laplacian eigenvalues. Therefore, $\kappa_i=\mu_{i+1}$.
\end{proof}

\begin{corollary}
\label{311}
Let $\Omega$ be a bounded Lipschitz domain in $\mathbb{R}^n$. If $\kappa_1<\mu_2$, then for any $c \in (\kappa_1,\mu_2)$, any solution $u_c$ to \eqref{cequation} satisfies
\begin{align}
\label{jinxia}
    \int_{\partial \Omega}u_c\, \rm{d} \sigma<0.
\end{align}
\end{corollary}
\begin{proof}
Let $w_1$ be as in the proof of Proposition \ref{gensufintro}. By Proposition \ref{gensufintro}, 
\begin{align*}
    \int_{\partial \Omega}u_c\, \rm{d} \sigma\le 0.
\end{align*}
If the equality is attained above, then by second Green identity, such $u_c$ is perpendicular to $w_1$ in $L^2(\Omega)$, and thus they are linearly independent. Let $L=span\{u_c, w_1\}$. Hence similar argument implies that 
\begin{align}
    \label{q4'}
\max_{u\in L\setminus \{0\}}\frac{\int_{\Omega} |\nabla u|^2 \, dx}{\int_{\Omega} u^2 \, dx} \le c< \mu_2(\Omega),
\end{align}which contradicts the definition of $\mu_2$ via min-max principle. This proves  \eqref{jinxia}.
\end{proof}

Next, we prove Proposition \ref{byproduct1}.
\begin{proof}[Proof of Proposition \ref{byproduct1}]
Let $v$ be an Neumann Laplacian eigenfunction of $\mu_2$. We first show that if $\kappa_1=\mu_2$, then $\int_{\partial \Omega}v\, \rm{d} \sigma=0$. Indeed, if this is not the case, then we let $\bar{v}=\frac{1}{P(\Omega)}\int_{\partial \Omega}v\, \rm{d} \sigma$ and $w=v-\bar{v}$. Hence $\int_{\partial \Omega} w\,\rm{d} \sigma=0$, and 
\begin{align}
\label{guiju}
    \frac{\int_{\Omega}|\nabla w|^2\, dx}{\int_{\Omega}w^2\, dx}=\frac{\int_{\Omega}|\nabla v|^2\, dx}{\int_{\Omega}(v^2+\bar{v}^2)\, dx}< \mu_2(\Omega),
\end{align}where we have used that $\int_{\Omega}v\,dx=0$. This leads to a contradiction to $\kappa_1=\mu_2$. Hence $v$ has boundary mean zero, and thus by Theorem \ref{existence}, \eqref{cequation} is solvable when $c=\mu_2$.

Conversely, if $u_{\mu_2}$ is a solution with positive mean on the boundary, then by Proposition \ref{gensufintro} which we have already proved, we have $\kappa_1=\mu_2$. 

It remains to show that their eigenspaces are the same with respect to the two different boundary conditions. Note that by hypothesis and Theorem \ref{existence}, any Neumann Laplacian eigenfunction of $\mu_2$ has boundary mean zero, and thus is automatically the boundary mean zero Laplacian eigenfunction of $\kappa_1$ with constant Neumann data. If there is another function $w$  which is also a boundary mean zero Laplacian eigenfunction of $\kappa_1$ while $\partial_{\nu} w\ne 0$ on $\partial \Omega$, then after scaling, $w$ is a solution to \eqref{cequation} for $c=\mu_2$, and thus $w$ and $u_{\mu_2}$ are the same up to a Neumann eigenfunction of $\mu_2$. Since \eqref{omegaproperty}, we have that $\int_{\partial \Omega}w\, \rm{d} \sigma=\int_{\partial \Omega}u_{\mu_2}\, \rm{d} \sigma$, which is positive by hypothesis. This contradicts the boundary mean zero property of $w$.
\end{proof}

Combining previous results, we can now easily derive Theorem \ref{maintheorem}.
\begin{proof}[Proof of Theorem \ref{maintheorem}]
That $\kappa_1\le \mu_2$ is a special case of Proposition \ref{gaojiebijiao}. If $\kappa_1=\mu_2$, then by Theorem \ref{main2}, when $c \in (0,\mu_2)$, $\int_{\partial \Omega}u_c\, \rm{d} \sigma>0$. Conversely, if for any $c \in (0,\mu_2)$, we have $\int_{\partial \Omega}u_c\, \rm{d} \sigma>0$. Then by Corollary \ref{311}, we derive $\kappa_1=\mu_2$. 

\eqref{zonghe} is also an immediate consequence of Theorem \ref{main2}, Corollary \ref{311} and continuity.
\end{proof}

We can now obtain exact constants $\kappa_1$ on balls and rectangles.
\begin{corollary}
If $\Omega$ is a rectangular box in $\mathbb{R}^n$, then $\kappa_1=\mu_2$ and their eigenspaces are the same, with respect to boundary mean zero Laplacian with constant Neumann data and Neumann Laplacian.
\end{corollary}
\begin{proof}
By Theorem \ref{zheng2}, $u_{\mu_2}$ has positive boundary mean, and hence by Proposition \ref{byproduct1}, $\kappa_1=\mu_2$ and their eigenspaces are the same  with respect to boundary mean zero Laplacian with constant Neumann data and Neumann Laplacian.
\end{proof}

\begin{corollary}
\label{ball}
Let $B_R$ be a ball of radius $R$ in $\mathbb{R}^n$. Then $\kappa_1=\mu_2$ and their eigenspaces are the same, with respect to the boundary mean zero Laplacian with constant Neumann data and Neumann Laplacian.
\end{corollary}
\begin{proof}
Let $c \in (0,\mu_2)$ and $u_c$ be the solution to \eqref{cequation}. By the explicit formula for $u_c$ constructed in the proof of Proposition \ref{ballmotivation}, we have $u_c>0$ for any $c \in (0,\mu_2]$. Therefore, by Theorem \ref{maintheorem}, we have that $\kappa_1=\mu_2$ when the domain is a ball.
\end{proof}

Another easy consequence is the Theorem \ref{iso}. That is, balls maximize $\kappa_1(\cdot)$ among all Lipschitz domains with fixed volume.
\begin{proof}[Proof of Theorem \ref{iso}]
Let $\Omega$ be a Lipschitz domain and $\Omega^\sharp$ be the ball with the same volume as that of $\Omega$. By Theorem \ref{maintheorem}, Corollary \ref{ball} and the seminal Szeg\"o-Weinberger theorem (see \cite{Szego} and \cite{Wein}) which states that balls maximize $\mu_2$ among Lipschitz domains with fixed volume, we have
\begin{align}
\label{en1}
    \kappa_1(\Omega)\le \mu_2(\Omega)\le \mu_2(\Omega^\sharp)=\kappa_1(\Omega^\sharp).
\end{align}
Since \cite{BP12} also gives the stabilty inequality of $\mu_2$ at balls as 
\begin{align}
\label{BP}
    |B|^{2/n}\mu_2(B)-|\Omega|^{2/n}\mu_2(\Omega) \ge C(n)A^2(\Omega),
\end{align}the quantitative version \eqref{qk1} of $\kappa_1$ at ball is thus as a consequence of \eqref{en1} and \eqref{BP}.
\end{proof}

Recall that in Theorem \ref{main2}, it has been shown that for any Lipschitz domain $\Omega$ in $\mathbb{R}^n$, $f(c):=c\int_{\partial \Omega}u_c\, \rm{d} \sigma$ is a strictly decreasing positive function with respect to $c \in (0,\kappa_1)$, with $f(0^+)=\tfrac{P^2(\Omega)}{|\Omega|}$. Hence it is natural to figure out the following number
\begin{align}
\label{c0def}
    c_0:=\sup\{c>0: f(c')>0\, \,\mbox{for any $0<c'<c$}\}.
\end{align}
If $\Omega$ is the domain such that $\kappa_1<\mu_2$, then by Theorem \ref{maintheorem}, we know that in this case $c_0=\kappa_1$. Concerning the case of $\Omega$ on which $\kappa_1=\mu_2$, we have the following partial answer, and the proof is similar to that of Theorem \ref{main2}. 

\begin{proposition}
\label{c0}
Let $\Omega$ be a bounded Lipschitz domain in $\mathbb{R}^n$ on which $u_{\mu_2}$ exists and has positive mean over $\partial \Omega$. Let $f(c):=c\int_{\partial \Omega}u_c\, \rm{d} \sigma$, where $u_c$ is a solution to \eqref{cequation}. Let $\tilde{\kappa}_2$ be next distinct boundary mean zero Laplacian eigenvalue of $\kappa_1$. Then $f(c)$ is a decreasing, positive function on $(0,\tilde{\kappa}_2)$. If $\tilde{\kappa}_2 \in [\mu_i,\mu_{i+1})$ for some $i$, then $c_0=\tilde{\kappa}_2$ and $f(c)<0$ when $c \in (\tilde{\kappa}_2, \mu_{i+1})$.
\end{proposition}

\begin{proof}
Let $E_2$ be the Neumann Laplacian eigenspace of $\mu_2$, and similar to \eqref{kappam}, we consider 
\begin{align}
    \label{ka1}
\kappa(m,\Omega):=\inf\Big\{\frac{\int_{\Omega}|\nabla u|^2\, dx+\frac{1}{m}\left(\int_{\partial \Omega}u\, \rm{d} \sigma\right)^2 }{\int_{\Omega}u^2 \, dx}: u\in H^1(\Omega)\setminus\{0\}, \int_{\Omega}u\phi_2=0, \,\forall\, \phi_2\in E_2\Big\}.
\end{align}
Since $\int_{\Omega}\phi_2\, dx=0$ for any $\phi_2\in E_2$, and hence $\lim_{m \rightarrow \infty}\kappa(m,\Omega)=0$. Also, by Proposition \ref{byproduct1}, $E_2$ is exactly the eigenspace of $\kappa_1$, and then by the variational characterization of $\kappa_i$, we have that $\lim_{m \rightarrow 0^+}\kappa(m,\Omega)=\tilde{\kappa}_2$.
Let
$$m_0:=\inf\{m>0: \kappa(m,\Omega)<\tilde{\kappa}_2\}. $$
For each $m>m_0$, there exists a minimizer $u_m$ at which the infimum in \eqref{ka1} is attained. By variational argument, 
\begin{align}
\label{EL}
    \int_{\Omega}\nabla u_m \nabla \phi\, dx+\frac{1}{m}\left(\int_{\partial \Omega}u_m\, dx\right)\left( \int_{\partial \Omega}\phi\, \rm{d}\sigma \right)-\kappa(m,\Omega)\int_{\Omega}u_m\phi\, dx=0,\quad \forall \phi \perp \phi_2.
\end{align}
Since $\int_{\partial \Omega}\phi_2=0$ due to $\kappa_1=\mu_2$, and since $u_m\perp \phi_2$, \eqref{EL} also holds when $\phi=\phi_2$. Therefore, $u_m$ satisfies
\begin{align}
    \label{um}
\begin{cases}
-\Delta u_m=\kappa(m,\Omega) u_m\quad &\mbox{in $\Omega$}\\
\frac{\partial u_m}{\partial \nu}=-\frac{1}{m}\int_{\partial \Omega}u_m\,\rm{d} \sigma\quad &\mbox{on $\partial \Omega$}\\
\int_{\Omega}u_m\phi_2\, dx=0.
\end{cases}    
\end{align}
Note that $\kappa(m,\Omega)$ is a strictly decreasing function with respect to $m$, when $m>m_0$. Hence for any $c \in (0,\tilde{\kappa}_2)$, there exists $m_c>m_0$ such that $\kappa(m_c,\Omega)=c$. Hence for $c \in (0,\tilde{\kappa}_2)$, $u_c:=\frac{m_c u_{m_c}}{\int_{\partial \Omega}u_{m_c}\, \rm{d} \sigma}$ is a solution to \eqref{cequation}. Therefore, 
\begin{align*}
    c\int_{\partial \Omega}u_c\, \rm{d} \sigma=m_c \kappa(m_c, \Omega)>0.
\end{align*}
To show that $f(c)$ is strictly decreasing when $c \in (0, \tilde{\kappa}_2)$, note first that when $0<c_1<c_2<\tilde{\kappa}_2$, $m_{c_1}>m_{c_2}>m_0$. Consider
\begin{align*}
    m\kappa(m,\Omega)=\inf\Big\{\frac{m\int_{\Omega}|\nabla u|^2\, dx+\left(\int_{\partial \Omega}u\, \rm{d} \sigma\right)^2 }{\int_{\Omega}u^2 \, dx}: u \in H^1(\Omega)\setminus\{0\}, \int_{\Omega}u\phi_2=0,\, \forall \, \phi \in E_2\Big\}.
\end{align*}
Since constant function is not a minimizer to the above functional, when $m$ decreases, $m\kappa(m,\Omega)$ also strictly decreases. Hence 
\begin{align*}
    m_{c_1}\kappa(m_{c_1},\Omega)> m_{c_2}\kappa(m_{c_2},\Omega),
\end{align*}which is equivalent to 
\begin{align*}
    c_1\int_{\partial \Omega}u_{c_1}\, \rm{d} \sigma> c_2\int_{\partial \Omega}u_{c_2}\, \rm{d} \sigma.
\end{align*}Therefore, we have proved that $f(c)$ is a decreasing, positive function for $c\in (0,\tilde{\kappa}_2)$. 

Next, if $\tilde{\kappa}_2$ lies in some $[\mu_i,\mu_{i+1})$, we show that $f(c)<0$ when $c \in (\tilde{\kappa}_2, \mu_{i+1})$. By Proposition \ref{gensufintro}, $f(c) \le 0$, otherwise $\tilde{\kappa}_2=\mu_{i+1}$. If $f(c)=0$, then $c$ is also a boundary mean zero Laplacian eigenvalue. However, this means that $\mu_i\le \tilde{\kappa}_2 <c< \mu_{i+1}$, which contradicts Proposition \ref{gaojiebijiao}.
\end{proof}


\section{The equation \eqref{cequation} on equilateral triangles}
In this section, we study \eqref{cequation} in the case when $\Omega$ is an equilateral triangle, and we will also obtain exact constant $\kappa_1$ on equilateral triangles.

We may assume that $\Omega$ has vertices $(-1,0)$, $(1,0)$ and $(0,\sqrt{3})$, and hence the Neumann eigenvalues are $\frac{4}{9}\pi^2(m^2+mn+n^2)$, where $m,n \in \mathbb{Z}$. Each eigenvalue has a symmetric mode and an anti-symmetric mode. All the modes were originally discovered in \cite{Lame}, see also the exposition in \cite{Mc}. The symmetric modes are
\begin{align}
    \label{sym}
T_{s}^{m,n}=&\cos\left(\frac{\pi l}{\sqrt{3}}(\sqrt{3}-y)\right)\cos\left(\frac{\pi(m-n)}{3}x\right)  \nonumber \\
&+ \cos\left(\frac{\pi m}{\sqrt{3}}(\sqrt{3}-y)\right)\cos\left(\frac{\pi(n-l)}{3}x\right)\nonumber\\
&+\cos\left(\frac{\pi n}{\sqrt{3}}(\sqrt{3}-y)\right)\cos\left(\frac{\pi(l-m)}{3}x\right),
\end{align}
and the anti-symmetric modes are
\begin{align}
    \label{ant}
T_{a}^{m,n}=&\cos\left(\frac{\pi l}{\sqrt{3}}(\sqrt{3}-y)\right)\sin\left(\frac{\pi(m-n)}{3}x\right)  \nonumber \\
&+ \cos\left(\frac{\pi m}{\sqrt{3}}(\sqrt{3}-y)\right)\sin\left(\frac{\pi(n-l)}{3}x\right)\nonumber\\
&+\cos\left(\frac{\pi n}{\sqrt{3}}(\sqrt{3}-y)\right)\sin\left(\frac{\pi(l-m)}{3}x\right).
\end{align}
In the above, $m+n+l=0$. Note that $T_s^{0,1}$ and $T_a^{0,1}$ are the fundamental modes of Neumann Laplacian. Clearly, anti-symmetric modes have mean zero over boundary. By direct computation,
\begin{align}
\label{equisatisfy}
    \int_{\partial \Omega}T_s^{0,1}\, \rm{d} \sigma=\int_{\partial \Omega}T_a^{0,1}\, \rm{d} \sigma=0.
\end{align}
Therefore by Theorem  \ref{existence}, we have
\begin{corollary}
\label{ext}
Let $\Omega$ be an equilateral triangle. Then when $c=\mu_2$, \eqref{cequation} admits a solution.
\end{corollary}

Motivated by \cite{Mc}, we can construct an explicit solution to \eqref{cequation} when $\Omega$ is an equilateral triangle and $c=\mu_2$. Then as a corollary of Proposition \ref{gensufintro}, we find the exact constant $\kappa_1(\Omega)$ when $\Omega$ is an equilateral triangle.
\begin{theorem}
\label{equilateralconstant}
Let $\Omega$ be an equilateral triangle with side length $a$. Then 
\begin{align*}
    \kappa_1=\mu_2=\frac{16\pi^2}{9a^2}.
\end{align*}
Moreover, their eigenspaces are the same, with respect to boundary mean zero Laplacian with constant Neumann data and the Neumann Laplacian.
\end{theorem}
\begin{proof}
Without loss of generality, we assume that $a=2$, and let $\Omega$ be the equilateral triangle with vertices with the same coordinates as before. We construct the following function:
\begin{align}
    \label{equisolution}
    u(x,y)=\frac{1}{\frac{2\pi}{3}\sin(\frac{\pi}{\sqrt{3}})}\left(2\cos(\pi x/\sqrt{3})\cos(\pi y/3))+\cos(\frac{\pi}{\sqrt{3}}-\frac{2}{3}\pi y)\right).
\end{align}
One can check that the above $u$ satisfies \eqref{cequation} for $c=\mu_2$, and $\int_{\partial \Omega}u_c\, \rm{d} \sigma>0$. Therefore by Proposition \ref{byproduct1}, we conclude that $\kappa_1=\mu_2$ and that they have the same eigenspaces.
\end{proof}

\begin{corollary}
Let $\Omega$ be a triangle and $\Omega^*$ be an equilateral triangle with the same perimeter as $\Omega$. Then $\kappa_1(\Omega)\le \kappa_1(\Omega^*)$.
\end{corollary}

\begin{proof}
According to \cite{LS09}, 
\begin{align}
\label{mu2tri}
    \mu_2(\Omega)\le \mu_2(\Omega^*).
\end{align}
Hence by Theorem \ref{equilateralconstant}, 
\begin{align*}
    \kappa_1(\Omega)\le \mu_2(\Omega)\le \mu_2(\Omega^*)=\kappa_1(\Omega^*).
\end{align*}
\end{proof}

The next proposition says that the equality of \eqref{planar} cannot be achieved at equilateral triangles.
\begin{proposition}
\label{jin1}
Let $\Omega$ be an equilateral triangle, $c \in (0,\mu_2]$, and $u_c$ be a solution to \eqref{cequation}. Then
\begin{align}
\label{stt}
    \lim_{c\rightarrow \mu_2^-}c\int_{\partial \Omega}u_c\, \rm{d} \sigma=\mu_2\int_{\partial \Omega}u_{\mu_2}\, \rm{d} \sigma >\frac{1}{2}\frac{P^2(\Omega)}{|\Omega|}.
\end{align}
\end{proposition}

\begin{proof}
The second strict inequality in \eqref{stt} is straightforward from computation, since the explicit construction of a solution to \eqref{cequation} when $c=\mu_2$ and $\Omega$ is an equilateral triangle is given, see \eqref{equisolution}. 

To prove the limit equality, we let $\kappa(m,\Omega)$, $u_m$ and $m_0$ be defined as in the proof of Theorem \ref{main2}. We first show that if $\Omega$ is an equilateral triangle, then $m_0>0$.  Indeed, let $m_k$ be a decreasing sequence such that $m_k \rightarrow m_0$ as $k\rightarrow \infty$. Then up to a subsequence, $u_{m_k}$ converges weakly to some $u\in H^1(\Omega)$. If $m_0=0$, then $u$ has mean zero over $\partial \Omega$, and the Reileigh quotient of $u$ is equal to $\kappa_1$. By Theorem \ref{equilateralconstant}, $\kappa_1=\mu_2$ and their eigenspaces are the same, and hence $u$ must be a linear combination of \eqref{sym} and \eqref{ant}. However, since $u_{m_k}$ is symmetric with respect to the three lines of symmetry of $\Omega$, so is the limit $u$ of $u_{m_k}$, while any linear combination of \eqref{sym} and \eqref{ant} is not so. Hence we reach a contradiction.

The above argument essentially shows that from right as $m\rightarrow m_0>0$, $\int_{\partial \Omega}u_m\, \rm{d} \sigma$ cannot have a subsequence converging to $0$. For any $c \in (0,\mu_2)$, we can find $m>m_0$ such that $\kappa(m,\Omega)=c$. Hence $u_c:=\tfrac{m u_{m}}{\int_{\partial \Omega}u_m\, \rm{d} \sigma}$ is a uniformly bounded sequence in $H^1(\Omega)$ for $m>m_0$, and such $u_c$ is symmetric with respect to the three lines of symmetry of $\Omega$. Hence up to a constant factor, any subsequence of $u_c$ must converge to the function given by \eqref{equisolution} weakly in $H^1(\Omega)$ and thus due to the $L^2$ trace embedding theory, the limit equality in \eqref{stt} is valid for a further subsequence. Due to the arbitrary choice of subsequence of $u_c$ as $ c\rightarrow \mu_2^-$, we conclude that the first equality in \eqref{stt} is valid.
\end{proof}

\section{Application in concentration breaking on a thermal insulation problem}
In this section, we prove Theorem \ref{q2main} and give quantitative estimates on breaking thresholds for balls and rectangular boxes.

\begin{proof}[Proof of Theorem \ref{q2main}]
Let $u_m$ be as in the hypothesis of Theorem \ref{q2main}, that is, $u_m$ is a function where the following infimum is attained:
\begin{align}
    \label{lambdam'}
\lambda_m=\inf\Big\{\frac{\int_{\Omega}|\nabla u|^2dx+\frac{1}{m}\left(\int_{\partial \Omega}|u|\rm{d} \sigma\right)^2}{\int_{\Omega}u^2dx}:u\in H^1(\Omega)\setminus \{0\}\Big\}.
\end{align}
We may assume that $u_m \ge 0$ and $\int_{\Omega}u_m^2\, dx=1$.

We first prove the first statement. By Theorem \ref{maintheorem}, $\kappa_1 \le \mu_2$, and thus $\kappa_1$ is strictly less than $\lambda_D$, the first eigenvalue of Dirichlet Laplacian. Since as $m\rightarrow 0$, \eqref{lambdam'} tends to $\lambda_D$, and as $m\rightarrow \infty$, \eqref{lambdam'} tends to $0$, we can find $m_0>0$ such that $\lambda_{m_0}=\kappa_1$.

Let $w$ be the function such that the infimum in \eqref{defofkappa} is achieved. Without loss of generality we assume that $\int_{\Omega}w^2\, dx=1$. We argue by contradiction. Suppose that $u_m>0$ on $\partial \Omega$, then consider $u_m+\epsilon w$ as a trial function for the $\lambda_m(\cdot)$ problem. Here $\epsilon$ is chosen such that $u_m+\epsilon w>0$ on $\partial \Omega$. Let 
\begin{align}
\label{jihao}
    R_m(f):=\frac{\int_{\Omega}|\nabla f|^2\,dx+\frac{1}{m}\left(\int_{\partial \Omega}|f|\,\rm{d} \sigma\right)^2}{\int_{\Omega}f^2\,dx}.
\end{align}
On the one hand, $R_m(u+\epsilon w)\ge \lambda_m$.
On the other hand, $u_m>0$ on $\partial \Omega$ implies that $\partial_\nu u_m$ is a constant on $\partial \Omega$, and thus $\int_{\Omega}\nabla u_m \nabla w \,dx=\lambda_m\int_{\Omega}u_mw\, dx$. Since $\lambda_m>\kappa_1$ as $m<m_0$, from above we have
\begin{align*}
    R_m(u_m+\epsilon w)=&\frac{\int_{\Omega}|\nabla u_m|^2dx+2\epsilon \lambda_m \int_{\Omega}u_mw\,dx+\epsilon^2\int_{\Omega}|\nabla w|^2 \, dx +\frac{1}{m}\left(\int_{\partial \Omega}u_m\,\rm{d} \sigma\right)^2}{\int_{\Omega}u_m^2dx+2\epsilon \int_{\Omega}u_mw\, dx+\epsilon^2\int_{\Omega}w^2\, dx}\\
    =&\frac{\lambda_m+2\epsilon \lambda_m \int_{\Omega}u_mw\,dx+\epsilon^2\kappa_1}{1+2\epsilon \int_{\Omega}u_mw\, dx+\epsilon^2}\\
    <& \frac{\lambda_m+2\epsilon \lambda_m \int_{\Omega}u_mw\,dx+\epsilon^2\lambda_m}{1+2\epsilon \int_{\Omega}u_mw\, dx+\epsilon^2}=\lambda_m.
\end{align*}This leads to a contradiction to the choice of $m_0$. Therefore, we conclude that when $m<m_0$, $u$ must vanish on some portion of $\partial \Omega$.

\medskip

Next, we prove the second statement. First, since $u_c>0$ for all $c \in (0,\mu_2)$, by Theorem \ref{maintheorem}, we have $\kappa_1=\mu_2$. Let $m_0$ be as above, then by what we have proved, when $m<m_0$, $u$ must vanish on some portion of $\partial \Omega$.

To show that $u_m>0$ when $m>m_0$, we first proceed the same argument as in the proof of Theorem \ref{main2} by considering the following auxiliary functional
\begin{align}
\label{auxfunctional}
    \kappa(m,\Omega):=\inf\Big\{\frac{\int_{\Omega}|\nabla u|^2\, dx+\frac{1}{m}\left(\int_{\partial \Omega}u\, \rm{d} \sigma\right)^2 }{\int_{\Omega}u^2 \, dx}: u \in H^1(\Omega) \setminus\{0\}\Big\}.  
\end{align}
 Therefore,
\begin{align*}
    \lim_{m\rightarrow \infty}\kappa(m,\Omega)=0\quad \mbox{and}\quad  \lim_{m\rightarrow 0}\kappa(m,\Omega)=\kappa_1=\mu_2.
\end{align*}
Let 
\begin{align}
    \label{m1}
m_a:=\inf\{m>0: \kappa(m,\Omega)<\mu_2\}.    
\end{align}
Since $\lambda(m,\Omega) \ge \kappa(m,\Omega)$, when $m>m_0$, $\kappa(m,\Omega)<\mu_2$, and hence $m_0 \ge m_a$. Given $m>m_a$, let $u_m'$ be the function at which the infimum in \eqref{auxfunctional} is attained. Then $u_m'$ satisfies the Euler-Lagrange equation
\begin{align}
\label{auxEL}
    \begin{cases}
    -\Delta u_m'=\kappa(m,\Omega) u_m' \quad & \mbox{in $\Omega$}\\
    \frac{\partial u_m'}{\partial \nu}=-\frac{1}{m}\int_{\partial \Omega}u_m' \, \rm{d} \sigma \quad &\mbox{on $\partial \Omega$}.
    \end{cases}
\end{align}
Since $\kappa(m,\Omega)<\mu_2$, $u_c:=\frac{mu_m'}{\int_{\partial \Omega}u_m' \, \rm{d} \sigma}$ is the unique solution to \eqref{cequation} for $c=\kappa(m,\Omega)$. Since $u_c>0$ on $\partial \Omega$ by hypothesis, $u_m'$ does not change sign on $\partial \Omega$. We may assume that $u_m'>0$ on $\partial \Omega$. If for some $m>m_0$, $u_m$ vanishes on some portion of $\partial \Omega$, then $u_m$ cannot satisfy \eqref{auxEL}, and then
\begin{align}
    \label{con1}
\lambda_m=\frac{\int_{\Omega}|\nabla u_m|^2\, dx+\frac{1}{m}\left(\int_{\partial \Omega}u_m\, \rm{d} \sigma\right)^2 }{\int_{\Omega}u_m^2 \, dx}>\frac{\int_{\Omega}|\nabla u_m'|^2\, dx+\frac{1}{m}\left(\int_{\partial \Omega}u_m'\, \rm{d} \sigma\right)^2 }{\int_{\Omega}(u_m')^2 \, dx}=\kappa(m,\Omega).
\end{align}
Since $u_m'>0$ on $\partial \Omega$, $|u_m'|=u_m'$ on $\partial \Omega$, and hence 
\begin{align}
    \label{con2}
\kappa(m,\Omega)=\frac{\int_{\Omega}|\nabla u_m'|^2\, dx+\frac{1}{m}\left(\int_{\partial \Omega}u_m'\, \rm{d} \sigma\right)^2 }{\int_{\Omega}(u_m')^2 \, dx}=\frac{\int_{\Omega}|\nabla u_m'|^2\, dx+\frac{1}{m}\left(\int_{\partial \Omega}|u_m'|\, \rm{d} \sigma\right)^2 }{\int_{\Omega}(u_m')^2 \, dx}\ge \lambda_m.
\end{align}
Hence \eqref{con1}-\eqref{con2} lead to a contradiction. Therefore, we have proved that when $m>m_0$, $u_m>0$.

\medskip

Now we prove the third statement. Let $c_0 \in (0,\mu_2)$ be as in Remark \ref{strictlypositive}. We may assume that $c_0<\kappa_1$. Let $m_0'>0$ be such that $\lambda_{m_0'}=c_0$, and $u_m'$ be the function as above. Hence when $m>m_0'$, $u_m'>0$ on $\partial \Omega$. A similar argument to the proof of the second statement implies that $u_m>0$ when $m>m_0'$.
\end{proof}


Note that from $\Omega \in \mathcal{F}$, the concentration breaking threshold $m_0$ is characterized as the number such that $\lambda_{m_0}=\mu_2$. Hence when $m>m_0$, 
\begin{align*}
    m\lambda_m=c\int_{\partial \Omega}u_c\, \rm{d}\sigma,
\end{align*}where $u_c$ is a solution to \eqref{cequation} with $c=\lambda_m<\mu_2$. Hence the following two corollaries, which give quantitative estimates for breaking thresholds for balls and rectangular boxes, are immediate from Theorem \ref{zheng1}, Theorem \ref{zheng2} and Theorem \ref{q2main}.
\begin{corollary}
\label{quan1}
Let $\Omega$ be a ball domain in $\mathbb{R}^n$ or a square in $\mathbb{R}^2$ and $m_0=m_0(\Omega)$ be the concentration breaking threshold in Theorem \ref{q2main}. Then
\begin{align}
\label{quantiball}
    m_0(\Omega)= \frac{n-1}{n\mu_2(\Omega)}\frac{P^2(\Omega)}{|\Omega|}.
\end{align}
\end{corollary}

\begin{corollary}
\label{quan2}
Let $\Omega$ be a rectangular box in $\mathbb{R}^n$ and $m_0=m_0(\Omega)$ be the concentration breaking threshold in Theorem \ref{q2main}. Then
\begin{align}
\label{quantiball}
    m_0(\Omega)\ge  \frac{n-1}{n\mu_2(\Omega)}\frac{P^2(\Omega)}{|\Omega|}.
\end{align}
\end{corollary}

\section{Numerical results}
In this section, we provide numerical results on isosceles triangles, regular polygons, ellipses and rhombuses.

In the following, we call an isosceles triangle super-equilateral triangle, if the aperture is bigger than or equal to $\frac{\pi}{3}$, and we call an isosceles triangle sub-equilateral triangle, if the aperture is less than or equal to $\frac{\pi}{3}$.  By aperture, we mean the angle between two equal sides. Let $u_c$ be the solution to \eqref{cequation} when $c \in (0,\mu_2)$, then numerical analysis suggests several very interesting observations as follows.

\begin{enumerate}
\item Let $\Omega$ be a regular polygon in $\mathbb{R}^2$. If the polygon has at least $4$ sides, then
    \begin{align}
    \label{denghao}
        \lim_{c \rightarrow \mu_2}c\int_{\partial \Omega} u_c\, \rm{d} \sigma =\frac{1}{2}\frac{P^2(\Omega)}{|\Omega|}. 
    \end{align}
In particular, by Proposition \ref{gensufintro}, this suggests that for regular polygons with at least 4 sides, $\kappa_1=\mu_2$. (Recall that we have shown that \eqref{denghao} is not true for equilateral triangles.)
\item Let $\Omega$ be a super-equilateral triangle, then 
\begin{align*}
     \lim_{c \rightarrow \mu_2}c\int_{\partial \Omega} u_c\, \rm{d} \sigma>\frac{1}{2}\frac{P^2(\Omega)}{|\Omega|}.
\end{align*}
Moreover, among all isosceles triangles with aperture bigger than or equal to $\frac{\pi}{3}$, the infimum of left hand side is achieved when $\Omega$ is an equilateral triangle, and the infimum is still strictly bigger than or equal to $\frac{1}{2}\frac{P^2(\Omega)}{|\Omega|}$. In particular, this suggests that for super-equilateral triangles, $\kappa_1=\mu_2$.

\item Let $\Omega$ be a sub-equilateral triangle. Then when $c$ is close to $\mu_2$, $\int_{\partial \Omega}u_c\, \rm{d} \sigma<0$, and thus $\kappa_1<\mu_2$ and \eqref{omegaproperty} cannot be satisfied for $k=2$.

\item Let $\Omega$ be a domain enclosed by an ellipse. Then
\begin{align*}
     \lim_{c \rightarrow \mu_2}c\int_{\partial \Omega} u_c\, \rm{d} \sigma\ge \frac{1}{2}\frac{P^2(\Omega)}{|\Omega|}.
\end{align*}
Moreover, among all elliptic domains, the infimum of left hand side is achieved when $\Omega$ is a disk. This suggests that for ellipses, $\kappa_1=\mu_2$.

\item Let $\Omega$ be a rhombus. Then
\begin{align*}
     \lim_{c \rightarrow \mu_2}c\int_{\partial \Omega} u_c\, \rm{d} \sigma\ge \frac{1}{2}\frac{P^2(\Omega)}{|\Omega|}.
\end{align*}
Moreover, among all rhombuses, the infimum of left hand side is achieved when $\Omega$ is a square. This also suggests that for rhombuses, $\kappa_1=\mu_2$.

\item For any isosceles triangle or rhombus $\Omega$, when $c$ is close to $\mu_2$, $\min_{\partial \Omega} u_c<0$. However, when $\Omega$ is a regular polygon with at least four sides, then $\min_{\partial \Omega}u_c>0$ for any $c \in (0,\mu_2)$. That is, regular polygons with at least four sides belong to the class $\mathcal{F}$ defined in \eqref{F}, while equilateral triangles do not.

\item Elliptic domains also belong to $\mathcal{F}$. More strikingly, among all elliptic domains in $\mathbb{R}^2$ with fixed area, $\inf_{c \in (0,\mu_2(\Omega))}\min_{\partial \Omega}u_c$ achieves the infimum only when $\Omega$ is a disk, and the infimum is positive.

\end{enumerate}

It would be very interesting to prove rigorously the above phenomena suggested by numerical simulation.

\section{Comparison of $\kappa_1$ and $\mu_2$ on sectors}

Recall that in the proof of Proposition \ref{byproduct1}, we show that a necessary condition for $\kappa_1=\mu_2$ to hold is that \eqref{omegaproperty} is true for $k=2$. It is natural to ask whether the converse is true. Even if the above claim is in general false, when focusing on domains with a line of symmetry, motivated by the isosceles triangular example, it is natural to further ask when all modes of $\mu_2$ are antisymmetric, whether we can derive $\kappa_1=\mu_2$. 

These questions are asked out of numerical results on various domains in section 8. In this section, by studying sector examples, we show that the answers to them are false, see Proposition \ref{disprovesector} below. 

We first stipulate some notations. Let $\alpha \in (0,2\pi)$ and $S(\alpha)$ be the sector
$$\{(x,y): 0\le r \le 1, \, |\theta|<\alpha/2\}.$$
Let $\alpha_0$ be the unique value such that
\begin{align}
    \label{alpha0}
j_{1,1}=j'_{\frac{\pi}{\alpha}, 1},   
\end{align}where $j_{s,1}$ denotes the first positive root of Bessel function of first kind of order $s$. 

Since $j_{1,1} \approx 3.8317$, numerical work reveals that $\alpha_0 \approx 1.1748$.

By a standard separation of variables argument, to find $\mu_2\left(S(\alpha)\right)$, one reduces to comparing the following two eigenvalues. First, one has the eigenvalue $j^2_{1,1}$ associated with the nonconstant radial mode $J_0(j_{1,1}r)$. Second, one has the eigenvalue $(j'_{\nu,1})^2$ associated with the angular mode $J_{\nu}(j'_{\nu,1}r)\sin(\nu\theta)$, where $\nu=\pi/\alpha$. This angular mode is odd with respect to the $x$-axis.

Then since $j'_{\nu,1}$ is strictly increasing with respect to $\nu$, we know that when $\alpha<\alpha_0$, $\mu_2\left(S(\alpha)\right)=j^2_{1,1}$. When $\alpha>\alpha_0$, $\mu_2\left(S(\alpha)\right)=(j'_{\frac{\pi}{\alpha},1})^2$. In both cases, $\mu_2$ is simple. Also, when $\alpha=\alpha_0$, $\mu_2$ has multiplicity 2, and there are both even (radial) and odd modes with respect to the $x$-axis.

\begin{lemma}
There exists $\alpha_1>\alpha_0$, such that when $0<\alpha <\alpha_1$, $\kappa_1\left(S(\alpha)\right)<\mu_2\left(S(\alpha)\right)$.
\end{lemma}

\begin{proof}
\begin{align*}
    \int_{\partial S(\alpha)}J_0(j_{1,1}r)\, ds=&2\int_0^1 J_0(j_{1,1}r)\, dr+\alpha J_0(j_{1,1})\\
    =& \frac{2}{j_{1,1}}\int_0^{j_{1,1}}J_0(r)\, dr+\alpha J_0(j_{1,1})\\
    \approx & 0.57009-0.40276\alpha\\
    >&0, \quad \mbox{when $\alpha<1.4154$.}
\end{align*}
In particular,
\begin{align}
    \label{fuhao1}
\int_{\partial S(\alpha)}J_0(j_{1,1}r)\, ds>0, \quad \mbox{when $\alpha\le \alpha_0\approx  1.1748$.}   
\end{align}
Note that
$$\phi:=J_0(j_{1,1}r)-\frac{1}{|\partial S(\alpha)|}\int_{\partial S(\alpha)}J_0(j_{1,1}r)\, ds$$is a valid trial function in the variational definition \eqref{defofkappa} of $\kappa_1$. When $\alpha\le \alpha_0$, since $J_0(j_{1,1}r)$ is a Neumann Laplacian eigenfunction of $\mu_2$ on $S(\alpha)$, and by \eqref{fuhao1} we have 
\begin{align*}
    \kappa_1\left(S(\alpha)\right)\le \frac{\int_{S(\alpha)}|\nabla J_0(j_{1,1}r)|^2 \, dx}{\int_{S(\alpha)}J_0^2(j_{1,1}r)\, dx+\frac{|S(\alpha)|}{|\partial S(\alpha)|^2}\left(\int_{\partial S(\alpha)}J_0(j_{1,1}r)\, ds\right)^2}<\mu_2\left(S(\alpha)\right).
\end{align*}
Hence when $\alpha$ is slightly larger than $\alpha_0$, the above inequality also holds. 
\end{proof}

\begin{proposition}
\label{disprovesector}
There are sectors such that $\kappa_1<\mu_2$, while all modes of $\mu_2$ are anti-symmetric.
\end{proposition}
\begin{proof}
By \cite[Lemma 5.1]{LS10} and the lemma above, we conclude that when $\alpha \in (\alpha_0,\alpha_1)$, all Neumann Laplacian modes of $\mu_2$ are odd with respect to $x$-axis, while $\kappa_1\left(S(\alpha)\right)<\mu_2\left(S(\alpha)\right)$ when $\alpha$ belongs to this range.
\end{proof}

Finally, we leave the following conjecture open.
\begin{conjecture}
\label{3+}
Let $\Omega$ be a symmetric domain in $\mathbb{R}^2$ with a line of symmetry. Suppose that $\mu_2$ has only even modes. Then $\kappa_1<\mu_2$.
\end{conjecture}

\section{Conflict of Interest}
On behalf of all authors, the corresponding author states that there is no conflict of interest.

\end{document}